\documentclass[11pt,a4paper]{article}
\usepackage[T1]{fontenc}
\usepackage[utf8]{inputenc}
\usepackage[english]{babel}
\usepackage{lmodern}
\usepackage{microtype}
\usepackage{geometry}
\usepackage{amsmath,amssymb,amsthm,mathtools}
\usepackage{booktabs}
\usepackage{enumitem}
\usepackage{xcolor}
\usepackage{hyperref}
\usepackage{cleveref}
\hypersetup{
 colorlinks=true,
 linkcolor=blue!45!black,
 citecolor=blue!45!black,
 urlcolor=blue!45!black,
 pdftitle={Smallest String Attractors and Minimal Coverage Certificates of Thue--Morse Words},
 pdfauthor={Simone Faro, Francesco Pio Marino, Arianna Pavone},
 pdfsubject={String attractors and minimal factor-coverages of finite Thue--Morse words},
 pdfkeywords={string attractor, Thue--Morse word, minimal unique substring, morphism, combinatorics on words}
}
\newtheorem{theorem}{Theorem}[section]
\newtheorem{lemma}[theorem]{Lemma}
\newtheorem{proposition}[theorem]{Proposition}
\newtheorem{corollary}[theorem]{Corollary}
\theoremstyle{definition}\newtheorem{definition}[theorem]{Definition}
\theoremstyle{remark}\newtheorem{remark}[theorem]{Remark}
\newcommand{\Occ}{\operatorname{Occ}}
\newcommand{\Cov}{\operatorname{Cov}}
\newcommand{\MUS}{\operatorname{MUS}}
\newcommand{\Acal}{\mathcal{A}}
\newcommand{\Hmin}{\mathcal{H}_{\min}}
\newcommand{\FL}{\mathcal{F}^{L}}
\newcommand{\FR}{\mathcal{F}^{R}}
\title{\textbf{Smallest String Attractors and Minimal Coverage Certificates of Thue--Morse Words}\\[0.5ex]
\large Extended Version}
\author{%
Simone Faro\thanks{Department of Mathematics and Computer Science, University of Catania, Catania, Italy.}
\and Francesco Pio Marino\thanks{Department of Mathematics and Computer Science, University of Catania, Catania, Italy; Univ Rouen Normandie, INSA Rouen Normandie, Universit\'e Le Havre Normandie, Normandie Univ, LITIS UR 4108, CNRS NormaSTIC FR 3638, IRIB, Rouen, France.}
\and Arianna Pavone\thanks{Department of Mathematics and Computer Science, University of Palermo, Palermo, Italy.}}
\date{}
\begin{document}
\maketitle

\begin{abstract}
String attractors provide a compact way of representing the complete factor structure of a word: a set of positions is an attractor if every distinct factor has at least one occurrence crossing one of the selected positions. Although the minimum attractor size is known for several classical families of words, describing \emph{all} optimal attractors is typically much more difficult, since it requires understanding the geometry of all factor occurrences rather than constructing a single optimal solution.

We give a complete description for the finite Thue--Morse words. Earlier work proved that four positions are necessary and sufficient for every order $n\geq 4$, but the collection of all smallest attractors remained unknown. For every $n\geq 6$, writing $h=2^{n-3}$, we prove that the smallest attractors are exactly the two reflected families
$
 \{2h+\varepsilon_1,3h+\varepsilon_2,4h+\varepsilon_3,6h+\varepsilon_4\}
$
and
$
 \{2h+\varepsilon_1,4h+\varepsilon_2,5h+\varepsilon_3,6h+\varepsilon_4\},
$
where the four offsets are chosen independently from $\{0,1\}$. Hence there are exactly $32$ smallest attractors for every $n\geq6$. The initial cases are genuinely exceptional: $t_5$ has $40$ smallest attractors and $t_4$ has $87$.

We also study the attractor condition independently of optimality. For every $n\geq5$, we characterize the complete antichain of inclusion-minimal factor coverages of $t_n$. It consists precisely of the coverages of $aa$, $bb$, and the eight minimal unique substrings of every generation $t_m$, $4\leq m\leq n$. Thus there are exactly $8n-22$ canonical constraints, forming an irredundant exact certificate for attractors of arbitrary cardinality.

When attention is restricted to four-position sets, this linear-size system collapses to a constant one: it is enough to test the $24$ minimal unique substrings coming from three consecutive generations, and sixteen of these already force the two optimal families. We also show that three generations are necessary within this natural consecutive-generation hierarchy.

The proofs combine the recursive boundary geometry of the Thue--Morse word, a trimmed recurrence for its minimal unique substrings, exact transport of occurrences across generations, and a desubstitution argument that identifies all inclusion-minimal coverage constraints.
\end{abstract}

\section{Introduction}

\subsection{String attractors and the classification problem}
A string attractor is a small set of positions that collectively ``see'' every distinct factor of a word. More precisely, a set of positions $\Gamma$ is a string attractor of a finite word $w$ if every nonempty factor $x$ of $w$ has at least one occurrence whose interval contains a position of $\Gamma$.

This definition is deceptively simple. A single selected position may attract many different factors at once, because every factor occurrence crossing that position is automatically represented. The problem is therefore highly global: choosing a good attractor requires understanding how occurrence intervals of many unrelated factors overlap.

String attractors were introduced by Kempa and Prezza as a common abstraction underlying several measures of repetitiveness and dictionary-compression schemes~\cite{KempaPrezza2018}. Their minimum size, denoted by $\gamma(w)$, is now a natural combinatorial invariant of a word. Determining $\gamma(w)$, however, answers only the first of two questions.

The first question is numerical: how many positions are necessary and sufficient to attract every factor? The second is structural: which sets of that minimum cardinality are actually attractors? The second problem is considerably more informative. Two word families may have the same minimum attractor size while exhibiting completely different numbers and geometries of optimal solutions. Knowing one optimal construction does not explain how rigid the optimum is, whether the selected positions are forced to lie near particular recursive boundaries, or how many independent choices remain.

The purpose of this paper is to answer both questions completely for the finite Thue--Morse words, and then to go one step further by describing the minimal constraints that characterize \emph{all} attractors, without restricting their cardinality.

\subsection{Why the Thue--Morse word?}
Let
$
\mu(a)=ab
$ 
and 
$
\mu(b)=ba
$
be the Thue--Morse morphism, and let $t_n=\mu^n(a)$. The words $t_n$ are among the most classical objects in combinatorics on words. They combine a very simple recursive definition with a surprisingly rich local factor structure.

Their self-similarity makes them a natural test case for repetitiveness measures. At the same time, that self-similarity does not make the attractor problem trivial. A factor may occur at several scales and at different locations inside the recursive decomposition. Consequently, a characterization of all smallest attractors cannot be obtained merely by inspecting the four canonical boundaries that appear in one known construction.

Kutsukake et al.~\cite{Kutsukake2020} proved the striking identity
$
\gamma(t_n)=4$
for $n\geq4$,
and exhibited a four-position attractor. Their theorem determines the optimum cardinality but does not characterize all optimal position sets.

More recently, Banbara et al.~\cite{Banbara2026} characterized all smallest attractors for Fibonacci and period-doubling words and explicitly left the analogous Thue--Morse problem open in the extended version of their work. Their computational experiments also indicate why a direct brute-force treatment is unattractive: a Walnut computation involving four free attractor positions exceeded the available memory. This suggests that a satisfactory solution should expose the recursive combinatorial structure behind the attractors rather than enumerate them directly.

We provide such a structural characterization.

\subsection{All smallest attractors}
Our first main result determines every smallest attractor of $t_n$. For $n\geq6$, write $h=2^{n-3}$, so that $|t_n|=8h$. The complete solution stabilizes from this point onward.

There are two families
\[
\FL_n=\bigl\{\{2h+\varepsilon_1,3h+\varepsilon_2,4h+\varepsilon_3,6h+\varepsilon_4\}:\varepsilon_1,\ldots,\varepsilon_4\in\{0,1\}\bigr\},
\]
and
\[
\FR_n=\bigl\{\{2h+\varepsilon_1,4h+\varepsilon_2,5h+\varepsilon_3,6h+\varepsilon_4\}:\varepsilon_1,\ldots,\varepsilon_4\in\{0,1\}\bigr\}.
\]
Every smallest attractor belongs to exactly one of these two families, and every member of the two families is an attractor.

The form of the answer has a useful geometric interpretation. The relevant positions occur in adjacent pairs $\{jh,jh+1\}$ around a small number of recursive boundaries. Within either family, one independently chooses one endpoint of each of four such boundaries. Thus each family contains $2^4=16$ attractors, giving $|\Acal(t_n)|=32$ for $n\geq6$.

The stabilization is not visible in the two preceding orders. We find $|\Acal(t_5)|=40$ and $|\Acal(t_4)|=87$. These exceptional cases are handled separately and also serve as useful checks on the stable theory.

\subsection{Beyond smallest attractors: the canonical coverage constraints}
The classification above concerns attractors of minimum cardinality. A different question is to understand the attractor condition itself.

For a factor $x$ of a word $w$, let $\Cov_w(x)$ denote the set of positions covered by at least one occurrence of $x$. Then
\[
\Gamma\text{ is an attractor of }w
\quad\Longleftrightarrow\quad
\Gamma\cap\Cov_w(x)\neq\varnothing
\]
for every nonempty factor $x$.

Most of these coverage constraints are redundant. If $\Cov_w(x)\subseteq\Cov_w(y)$, then every set hitting the coverage of $x$ automatically hits the coverage of $y$. The essential constraints are therefore the inclusion-minimal members of
\[
\mathcal H(w)=\{\Cov_w(x):x\text{ is a nonempty factor of }w\}.
\]
We determine these constraints exactly for Thue--Morse words.

For every $n\geq5$, the inclusion-minimal factor coverages of $t_n$ are precisely $\Cov_{t_n}(aa)$ and $\Cov_{t_n}(bb)$, together with $\Cov_{t_n}(M)$ for every $M\in\MUS(t_m)$ and $4\leq m\leq n$. Since every generation $t_m$, $m\geq4$, has eight minimal unique substrings, the total number of minimal constraints is
\[
2+8(n-3)=8n-22.
\]
This result is conceptually different from the classification of the smallest attractors. It characterizes the exact constraint system defining \emph{every} attractor of $t_n$, irrespective of cardinality. These constraints form an antichain and none can be removed.

\subsection{Why minimal unique substrings appear}
Minimal unique substrings, or MUSs, play a central role in the proof. A MUS is a factor that occurs exactly once but becomes repeated as soon as either end is removed. Since its occurrence is unique, every attractor must cross that occurrence. Thus MUSs provide natural necessary constraints.

For Thue--Morse words, Kishi et al.~\cite{Kishi2026} gave a closed characterization of all MUSs. We use that result as an external structural ingredient. What is less immediate is how MUSs belonging to different generations behave when they are all viewed inside the same larger word $t_n$. A factor that is unique in its source word $t_m$ generally acquires several occurrences in $t_n$. The coverage relevant to the attractor problem is the union of \emph{all} those occurrences.

The key technical mechanism is a particularly simple recurrence between consecutive generations. The MUS of generation $m+1$ corresponding to a MUS $M$ of generation $m$ is obtained by applying the Thue--Morse morphism and deleting the first and last symbols:
$
M^{(m+1)}=\operatorname{trim}\!\left(\mu(M^{(m)})\right).
$
We prove an exact transport rule for their occurrence starts, valid even when the factors are considered inside later ambient generations. This allows coverage information at arbitrary scale to be derived from a constant amount of base data.

\subsection{From \texorpdfstring{$8n-22$}{8n-22} constraints to a constant certificate}
The full canonical system contains $8n-22$ constraints, which is optimal when attractors of arbitrary cardinality are considered. Smallest attractors, however, contain exactly four positions. This severe cardinality restriction introduces additional geometric rigidity.

We prove that, for four-position sets, it is enough to test only the MUSs belonging to three consecutive generations:
$
\MUS(t_n)\cup\MUS(t_{n-1})\cup\MUS(t_{n-2}).
$
There are exactly $24$ such factors. Thus, for every $n\geq6$, a four-position set is an attractor if and only if it intersects the ambient coverage of each of these $24$ factors. This depth-three statement is best possible within the natural hierarchy obtained by taking consecutive MUS generations: the MUSs of $t_n$ and $t_{n-1}$ alone do not suffice.

The proof of the complete classification needs even less. We isolate sixteen carefully chosen MUS-derived coverage sets. Four of them are pairwise disjoint, so a four-position hitting set is forced to place exactly one position in each of four regions. The remaining witnesses progressively restrict those positions until only the two families $\FL_n$ and $\FR_n$ remain. This sixteen-factor transversal argument is the combinatorial core of the classification.

\subsection{Proof ideas}
Two different recursive mechanisms underlie the two main results.

For the characterization of the smallest attractors, the crucial object is the recursive boundary geometry of the Thue--Morse word. We prove a centered-occurrence property: every factor of length at least two has an occurrence containing both endpoints of one of four distinguished recursive boundaries. Consequently, any set choosing one endpoint of each of the appropriate four boundaries attracts every factor of length at least two; the one-letter factors are then checked separately. This gives sufficiency.

Necessity is proved through transported MUS coverages. Their geometry produces the sixteen-factor witness system described above and forces every four-position attractor into one of the two stable families.

For the canonical coverage theorem, the argument is different. We first show that an inclusion-minimal coverage admits a shortest representative that is \emph{two-sided frequency-critical}: removing either its first or its last symbol strictly increases the number of occurrences. A synchronization argument then gives a unique morphic phase for every sufficiently long factor. Critical factors can therefore be desubstituted, and criticality is preserved under the descent.

Iterating this process reduces every long critical factor to one of the eight critical factors of length four, which are exactly the MUSs of $t_4$. Reversing the descent recovers the MUS hierarchy across all generations. Finally, exact occurrence geometry shows that these MUS-derived coverages, together with the coverages of $aa$ and $bb$, are pairwise incomparable. This proves that they are exactly the inclusion-minimal factor coverages.

\subsection{Main contributions}
Our results can be summarized as follows.
\begin{enumerate}[leftmargin=*]
\item We completely characterize the smallest string attractors of every finite Thue--Morse word. For $n\geq6$, there are exactly two stable families and $32$ solutions; $t_5$ and $t_4$ are exceptional, with $40$ and $87$ solutions, respectively.
\item For every $n\geq5$, we determine all inclusion-minimal factor coverages of $t_n$. There are exactly $8n-22$, namely the coverages of $aa$, $bb$, and the eight MUSs of every generation from $4$ through $n$. They form the canonical irredundant exact certificate for attractors of arbitrary cardinality.
\item For four-position sets, we obtain a constant-size exact certificate consisting of the $24$ MUSs of three consecutive generations. We show that three generations are necessary within this consecutive-generation hierarchy.
\item We isolate a sixteen-factor witness system whose coverage geometry forces precisely the two stable families of smallest attractors.
\item Using the MUS characterization of Kishi et al., we derive a Thue--Morse-specific trimmed recurrence and an exact occurrence-transport rule across generations. These provide the bridge from finite base data to the uniform coverage formulas used throughout the paper.
\end{enumerate}

\subsection{Organization of the paper}
Section~\ref{sec:related-work} reviews the literature on string attractors, Thue--Morse words, automatic methods, and minimal unique substrings, with particular attention to the results on which our proofs build. Section~3 introduces notation and the structural properties of finite Thue--Morse words needed later, including the recursive boundary geometry that supports the sufficiency argument. Section~4 develops the transition from the known MUS structure to exact occurrence coverages. Section~5 proves the stable characterization of all smallest attractors. Section~6 turns from smallest attractors to arbitrary attractors and characterizes the complete family of inclusion-minimal factor coverages. The remaining sections treat the exceptional words $t_4$ and $t_5$, describe the independent computational verification used as a consistency check, and discuss the broader implications of the two characterizations.

\section{Related work}
\label{sec:related-work}

The present paper sits at the intersection of three lines of work: string attractors as measures of repetitiveness, exact attractor questions for structured word families, and the combinatorics of minimal unique substrings. We summarize only the aspects directly used or compared with our results.

\subsection{String attractors and repetitive structure}
String attractors were introduced by Kempa and Prezza~\cite{KempaPrezza2018} as a unifying abstraction for several dictionary-compression mechanisms. Their definition isolates a simple combinatorial requirement---every distinct factor must have an occurrence crossing a selected position---while still capturing lower bounds for a range of compressed representations.

The subsequent combinatorial study of attractors showed that classical word families provide a useful laboratory for understanding how the minimum attractor size interacts with factor structure. Mantaci et al.~\cite{Mantaci2019,Mantaci2021} studied attractors from this point of view for several families, including Thue--Morse words. More recently, attractor-based complexity functions have also been investigated for infinite words~\cite{Cassaigne2024}. These directions motivate treating $\gamma(w)$ not only as an algorithmic parameter but also as a combinatorial invariant.

At a more algorithmic level, B\'eal, Crochemore, and Romana~\cite{Beal2025} gave a linear-time attractor checker based on suffix automata together with a greedy construction procedure. Such algorithms answer the decision or construction problem for a given word; our goal is different, namely to describe the entire family of optimum position sets for one highly structured infinite family of finite words.

\subsection{Attractors of Thue--Morse and related morphic words}
For the canonical finite Thue--Morse words, Mantaci et al.~\cite{Mantaci2019} initially conjectured that the minimum attractor size grows with the order. Kutsukake et al.~\cite{Kutsukake2020} disproved this by proving the constant identity $\gamma(t_n)=4$ for all $n\ge4$ and by giving a four-position construction based on recursive boundaries.

Dolce~\cite{Dolce2023} subsequently considered arbitrary factors of the infinite Thue--Morse word. His result shows that every such factor admits a small attractor and establishes a sharp upper bound of five. This is complementary to the present problem: here the ambient words are the canonical prefixes $t_n$, and the objective is to classify all optimum position sets rather than to construct one attractor for each factor.

The work most directly connected with our classification is that of Banbara et al.~\cite{Banbara2026}. They determined all smallest attractors of finite Fibonacci and period-doubling words and used the number and geometry of smallest attractors to distinguish families sharing the same minimum cardinality. Their extended version explicitly lists the corresponding Thue--Morse classification as an open problem. It also reports that a direct Walnut computation with four free attractor positions exceeded the available memory, reinforcing the need for a structural treatment.

Automatic methods provide a complementary viewpoint. Schaeffer and Shallit~\cite{Schaeffer20} formulated attractor properties of automatic sequences in first-order logic and studied them computationally using Walnut~\cite{Mousavi2016}. We use neither automatic decision procedures nor their structural claims in our proofs; our arguments instead exploit the recursive block geometry and explicit factor coverages of finite Thue--Morse words.

\subsection{Minimal unique substrings and occurrence transport}
Minimal unique substrings have been studied independently of attractors, both algorithmically and combinatorially; see, for example, Ilie and Smyth~\cite{IlieSmyth2011}. Their relevance here is immediate: a MUS has a unique occurrence, hence every attractor must cross that occurrence interval.

Kishi et al.~\cite{Kishi2026} recently introduced a general framework of occurrence-preserving morphisms and applied it to finite Fibonacci and Thue--Morse words. In particular, they obtained the closed characterization of the eight MUSs in every sufficiently large finite Thue--Morse generation. This theorem is an external structural ingredient of our proofs.

Our occurrence-transport lemma should be read in that context but not as a generalization of their framework. Their position-level correspondence concerns ordinary morphic images under suitable interference-free hypotheses. The factors needed here are instead the actual consecutive-generation Thue--Morse MUSs, obtained by deleting one symbol from each end of the morphic image. We prove the exact start-position correspondence for these specific trimmed factors because their full ambient coverages are what enter the attractor witness system.

\subsection{Position of the present results}
The previous results determine the minimum cardinality for $t_n$, construct individual attractors, characterize MUSs, or classify smallest attractors for other morphic families. The two questions addressed here remain distinct from those contributions. First, we classify \emph{every} smallest attractor of every finite Thue--Morse word. Second, we characterize the complete family of inclusion-minimal factor coverages and hence the canonical irredundant constraint system defining attractors of arbitrary cardinality. The latter result goes beyond the smallest-attractor classification and explains why the MUS hierarchy is the correct source of constraints across scales.

\section{Preliminaries and structural intuition}
\label{sec:preliminaries}

This section fixes the notation used throughout the paper and, more importantly, develops the geometric picture behind the proofs. The formal definitions are standard, but two viewpoints will be used repeatedly: factors are represented by the \emph{sets of positions covered by their occurrences}, and the Thue--Morse word is read through a hierarchy of recursively nested blocks. Keeping these two pictures in mind makes the later witness arguments considerably easier to follow.

\subsection{Words, occurrences, and coverage sets}
All positions are one-based. For integers $i\le j$, we write
$
[i,j]=\{i,i+1,\ldots,j\}.
$
If $w$ is a finite word, then $|w|$ denotes its length and $w[i..j]$ the factor from position $i$ through position $j$.

For a factor $x$ of $w$, let
$
\Occ_w(x)=\{p:w[p..p+|x|-1]=x\}
$
be the set of starting positions of its occurrences. Each occurrence beginning at $p$ occupies the interval
$
[p,p+|x|-1].
$
The \emph{coverage} of $x$ in $w$ is the union of all these intervals:
\begin{equation}
\Cov_w(x)=\bigcup_{p\in\Occ_w(x)}[p,p+|x|-1].
\label{eq:coverage-definition}
\end{equation}

The distinction between an occurrence and a coverage is important. A factor can have many occurrences, but from the attractor point of view only the union of the positions covered by those occurrences matters. If a selected position lies in $\Cov_w(x)$, then at least one occurrence of $x$ crosses that position.

\begin{definition}[String attractor]
A set $\Gamma\subseteq[1,|w|]$ is a \emph{string attractor} of $w$ if
$
\Gamma\cap\Cov_w(x)\neq\varnothing
$
for every nonempty factor $x$ of $w$.
\end{definition}

Equivalently, the attractor problem is a hitting-set problem whose universe is the set of positions of $w$ and whose constraints are the factor coverages. This simple reformulation is the bridge between the combinatorics of factors and the geometric arguments used later.

\paragraph{A small example.}
Consider $w=ababa$ and the factor $x=aba$. It occurs at positions $1$ and $3$, hence
$
\Cov_w(aba)=[1,3]\cup[3,5]=[1,5].
$
Thus every nonempty set of positions hits the coverage of $aba$, despite the fact that $aba$ has only two occurrences. By contrast, a factor with a unique occurrence contributes a single interval and can impose a much more localized constraint.

\subsection{Minimal coverages as the essential constraints}
Let
$
\mathcal H(w)=\{\Cov_w(x):x\text{ is a nonempty factor of }w\},
$
and let $\Hmin(w)$ denote the inclusion-minimal members of $\mathcal H(w)$.

The relevance of $\Hmin(w)$ is immediate but fundamental. If $C,D\in\mathcal H(w)$ satisfy $C\subseteq D$, then every set hitting $C$ also hits $D$. Hence $D$ is redundant as a constraint once $C$ is present.

\begin{proposition}[Minimal-coverage reduction]
\label{prop:minimal-coverage-reduction}
For every finite word $w$ and every $\Gamma\subseteq[1,|w|]$,
\[
\Gamma\text{ is a string attractor of }w
\quad\Longleftrightarrow\quad
\Gamma\cap C\neq\varnothing
\text{ for every }C\in\Hmin(w).
\]
\end{proposition}

\begin{proof}
The forward implication follows directly from the definition. Conversely, $\mathcal H(w)$ is finite, so every member of $\mathcal H(w)$ contains at least one inclusion-minimal member. A set hitting every element of $\Hmin(w)$ therefore hits the coverage of every factor.
\end{proof}

This proposition motivates one of the two main questions of the paper. Instead of asking only for minimum-cardinality attractors, we can ask for the smallest possible \emph{constraint system} that defines the attractor property. For Thue--Morse words we will identify this system exactly.

\subsection{The finite Thue--Morse words}
We work over the binary alphabet $\{a,b\}$. Let
$
\mu(a)=ab$ and
$
\mu(b)=ba,
$
and define
\[
t_0=a,\qquad t_n=\mu^n(a).
\]
We write $\overline{w}$ for the binary complement of a word $w$, obtained by exchanging $a$ and $b$. Then
$
|t_n|=2^n
$
and the basic recursion is
\begin{equation}
t_n=t_{n-1}\,\overline{t_{n-1}}.
\label{eq:tm-recursion}
\end{equation}
Applying the recursion twice gives the four-block decomposition
\begin{equation}
t_n=t_{n-2}\,\overline{t_{n-2}}\,\overline{t_{n-2}}\,t_{n-2}.
\label{eq:tm-four-block}
\end{equation}

For orientation, the first few words are
\[
\begin{aligned}
t_0&=a,\\
t_1&=ab,\\
t_2&=abba,\\
t_3&=abbabaab,\\
t_4&=abbabaabbaababba.
\end{aligned}
\]
The precise letters will quickly become less important than the recursive block structure. Equations~\eqref{eq:tm-recursion} and~\eqref{eq:tm-four-block} let us locate a factor relative to large canonical blocks and then descend to smaller scales.

\subsection{The decomposition tree viewpoint}
It is useful to picture $t_n$ as the root of a perfect binary decomposition tree. A node labelled by $t_k$ has two children labelled by $t_{k-1}$ and $\overline{t_{k-1}}$; a node labelled by $\overline{t_k}$ has the complementary pair of children. The leaves correspond to individual positions of $t_n$.

Now take an occurrence interval $[i,j]$ with $i<j$. The lowest node whose block contains both endpoints has the property that $i$ lies in its left child and $j$ lies in its right child. Therefore the interval necessarily crosses the boundary between those two children.

This observation is elementary, but it is the reason recursive boundaries enter the attractor problem at all: every factor occurrence of length at least two crosses the center boundary of some canonical block. The real difficulty is to show that the same factor can be realized at one of only a few \emph{global} boundaries, independent of the factor. That is the content of the centered-occurrence property proved later.

\subsection{Four canonical boundaries}
For $n\ge4$, define
\begin{equation}
K_n=\left\{2^{n-2},\;3\cdot2^{n-3},\;2^{n-1},\;3\cdot2^{n-2}\right\}.
\label{eq:canonical-boundaries}
\end{equation}
A number $c\in K_n$ represents the boundary $c\mid c+1$.

These four boundaries arise at different levels of the recursive decomposition. Two are the centers of the two halves of $t_n$, one is the global center of $t_n$, and one comes from the four-block decomposition~\eqref{eq:tm-four-block}. Repeatedly descending by two generations preserves the same center, so copies of smaller $t_k$ or $\overline{t_k}$ reappear centered at one of these four locations.

This is the geometric origin of the classical four-position attractor. Kutsukake et al.~\cite{Kutsukake2020} proved that
$
\gamma(t_n)=4$
with
$
n\ge4
$
and exhibited the set $K_n$ itself as an attractor. In our classification, the four canonical positions are not rigid: from order six onward, each can be shifted to the adjacent endpoint of its boundary in a controlled way, and reflection creates a second stable pattern. The later proofs explain exactly which endpoint choices remain possible.

\subsection{A preview of the centered-occurrence property}
The following strengthening of the boundary picture will be proved in the structural section of the paper.

\begin{theorem}[Centered-occurrence property --- preview]
\label{thm:centered-preview}
Let $n\ge4$ and let $x$ be a factor of $t_n$ with $|x|\ge2$. Then there exist $c\in K_n$ and an occurrence of $x$ whose interval contains both $c$ and $c+1$.
\end{theorem}

The difference between ``crosses the boundary'' and ``contains both endpoints'' matters. If an occurrence contains both $c$ and $c+1$, then it is hit regardless of which endpoint of the pair we select. This is precisely what allows the $\varepsilon_i\in\{0,1\}$ freedom in the two optimal families.

The proof uses a center-descent argument. Roughly, a canonical copy of $t_k$ centered at a boundary contains a canonical copy of $t_{k-2}$ centered at the same boundary. Consequently, every block type appearing in the decomposition tree has a centered representative at one of the four locations in $K_n$. An arbitrary occurrence can then be transferred, with the same relative coordinates, to that centered representative.

\subsection{Minimal unique substrings as localized constraints}
A factor $M$ of a finite word $w$ is a \emph{minimal unique substring} (MUS) if it occurs exactly once and every proper factor of $M$ occurs at least twice. Equivalently, both its longest proper prefix and its longest proper suffix are repeated.

MUSs are particularly natural for attractors. Since a MUS has a unique occurrence, every attractor must select at least one position inside that occurrence interval. In other words, every MUS supplies a necessary localized hitting constraint.

For finite Thue--Morse words, Kishi et al.~\cite{Kishi2026} obtained a closed description of all MUSs. In the indexing used here, for every $m\ge4$ there are exactly eight of them, each of length
$
2^{m-3}+2.
$
Their explicit form will be stated when we begin the occurrence-transport analysis.

Two cautions are useful from the outset. First, a MUS is unique only in its \emph{source word}. If a MUS of $t_m$ is considered as a factor of a larger word $t_n$, with $n>m$, it may have several occurrences. Its relevant constraint is then its full ambient coverage $\Cov_{t_n}(M)$. Second, the eight MUSs of the current generation alone do not capture all four-position attractors: the proof will need information transported from earlier generations.

\subsection{Three levels of constraints}
The results of the paper can now be viewed as a hierarchy of increasingly compressed hitting systems.

At the most general level, an attractor must hit the coverage of every factor. Proposition~\ref{prop:minimal-coverage-reduction} reduces this to the inclusion-minimal coverages. For $t_n$ we will show that there are exactly $8n-22$ of them.

If we additionally know that $|\Gamma|=4$, the geometry becomes much more rigid. Then the complete test can be reduced to the $24$ MUSs of three consecutive generations. Finally, the classification proof itself uses only sixteen carefully selected witnesses.

Thus the proof moves through the chain
\[
\begin{aligned}
\text{all factor coverages}
&\Longrightarrow 8n-22\text{ canonical constraints}\\
&\Longrightarrow 24\text{ three-generation MUS constraints}\\
&\Longrightarrow 16\text{ forcing witnesses}.
\end{aligned}
\]
The arrows should not be interpreted as unconditional logical equivalences: the last two reductions exploit the assumption that the candidate set has four positions. They are best read as a roadmap for how the combinatorial structure becomes progressively more rigid near optimum size.

\subsection{Notation used later}
For convenience, we collect the principal notation here.
\begin{center}
\begin{tabular}{@{}ll@{}}
\toprule
Notation & Meaning \\
\midrule
$t_n$ & finite Thue--Morse word $\mu^n(a)$ \\
$\overline{w}$ & binary complement of $w$ \\
$\Occ_w(x)$ & starting positions of occurrences of $x$ in $w$ \\
$\Cov_w(x)$ & positions covered by all occurrences of $x$ in $w$ \\
$\gamma(w)$ & minimum cardinality of a string attractor of $w$ \\
$\Acal(w)$ & family of attractors of cardinality $\gamma(w)$ \\
$\MUS(w)$ & set of minimal unique substrings of $w$ \\
$\mathcal H(w)$ & family of all factor coverages of $w$ \\
$\Hmin(w)$ & inclusion-minimal members of $\mathcal H(w)$ \\
$K_n$ & four canonical boundary positions in~\eqref{eq:canonical-boundaries} \\
\bottomrule
\end{tabular}
\end{center}

The next section turns from this geometric picture to the exact occurrence structure of the MUS hierarchy. The goal is to replace repeated large-scale calculations by a constant amount of base data plus a precise transport rule across generations.

\section{From MUS structure to occurrence coverages}
\label{sec:mus-transport}

The previous section explained why minimal unique substrings are natural localized constraints for string attractors. We now address the main technical difficulty that arises when these constraints are compared across different Thue--Morse generations.

A MUS of $t_m$ is unique in its source word, but it is usually no longer unique when viewed inside a larger word $t_{m+r}$. For the attractor problem, we therefore cannot keep only its original occurrence: we need the complete set of its occurrences in the ambient word, and hence its full coverage. Recomputing those occurrences independently at every order would hide the recursive structure that makes a uniform proof possible.

The goal of this section is to show that the relevant occurrence information can instead be transported exactly from one generation to the next. There are three ingredients. First, sufficiently long morphic images have a forced alignment with the canonical Thue--Morse blocks. Second, consecutive MUSs are not ordinary morphic images: one symbol is deleted at each end. Third, despite this trimming, occurrence starts obey a particularly simple doubling rule. Together, these facts reduce all later coverage calculations to a small table at the first few orders.

\subsection{The MUS hierarchy}
\label{subsec:mus-hierarchy}

For a word $u$, define
$
\operatorname{EXT}(u)=\{aua,aub,bua,bub\}.
$
Kishi et al.~\cite{Kishi2026} characterize the MUSs of finite Thue--Morse words in a closed form. Translated to our indexing, their theorem gives the following statement.

\begin{theorem}[Kishi et al.]
\label{thm:kishi-mus}
For every $m\ge4$,
$
\MUS(t_m)=\operatorname{EXT}(t_{m-3})\cup
\operatorname{EXT}(\overline{t_{m-3}}).
$
In particular, $t_m$ has exactly eight minimal unique substrings, each of length
$
2^{m-3}+2.
$
\end{theorem}

Thus every MUS has a long Thue--Morse core surrounded by one letter on each side. The core doubles in length from one generation to the next. This is already suggestive of a morphic recurrence, but one small point matters: applying $\mu$ to a MUS also doubles its two boundary letters, whereas the next MUS contains only one new boundary letter at each end. The correct recurrence is therefore a \emph{trimmed} morphic image.

For a word $z$ with $|z|\ge3$, write
$
\operatorname{trim}(z)=z[2..|z|-1].
$
For every $m\ge4$, order the eight members of $\MUS(t_m)$ by the starting positions of their unique occurrences in $t_m$, and denote them by
$
M_1^{(m)},\ldots,M_8^{(m)}.
$
This ordering will be preserved by the recurrence below.

\subsection{Synchronization: why the morphic phase is forced}
\label{subsec:synchronization}

The Thue--Morse morphism has constant length two. Consequently, the word $t_N=\mu(t_{N-1})$ comes with a canonical partition into two-letter image blocks, each equal to $ab$ or $ba$. The first question is whether an occurrence of a word of the form $\mu(U)$ could be shifted by one position and still look like a valid morphic image.

For $|U|\ge2$, the answer is no. The reason is local. A shifted copy of two consecutive image blocks would cross two consecutive canonical block boundaries. For both crossing pairs to be valid Thue--Morse image blocks, three consecutive preimage letters would have to be equal, but the Thue--Morse word contains neither $aaa$ nor $bbb$.

\begin{lemma}[Synchronization]
\label{lem:synchronization}
Let $U$ be a binary word with $|U|\ge2$. Every occurrence of $\mu(U)$ in a finite Thue--Morse word starts at an odd position.
\end{lemma}

\begin{proof}
First observe that no finite Thue--Morse word contains $aaa$ or $bbb$. Indeed, every interval of three consecutive positions contains a complete canonical image block, and every such block is either $ab$ or $ba$.

Suppose for a contradiction that an occurrence of $\mu(U)$ in $t_N=\mu(t_{N-1})$ starts at an even position $2j$. Because $|U|\ge2$, the first two length-two blocks of this occurrence straddle two consecutive boundaries between canonical images. Let $z_j,z_{j+1},z_{j+2}$ be the corresponding three consecutive preimage letters.

The pair across the first canonical boundary is
$
\overline{z_j}\,z_{j+1},
$
where $\overline{z_j}$ denotes the complement of $z_j$. In order for this pair to equal one of the valid image blocks $ab$ or $ba$, its two letters must differ; over the binary alphabet this is equivalent to $z_j=z_{j+1}$. The next crossing pair similarly forces $z_{j+1}=z_{j+2}$. Hence the preimage would contain three consecutive equal letters, a contradiction.

The same argument covers occurrences adjacent to the ends of the finite word: an even start is never the first position, and an occurrence of length at least four contains the two crossing pairs needed above.
\end{proof}

\paragraph{Why this lemma matters.}
Synchronization turns a local occurrence in the large word into an unambiguous occurrence in the preimage. Once the phase is known, aligned pairs can be decoded uniquely. This is the mechanism that will rule out ``extra'' occurrences of a trimmed MUS that do not come from the previous generation.

\subsection{The trimmed MUS recurrence}
\label{subsec:trimmed-lifting}

We can now make the recurrence between consecutive MUS generations exact. The statement has two parts: a factor identity and an occurrence identity. The factor identity says which MUS appears at the next order; the occurrence identity says that this relation remains valid even when both factors are viewed inside later ambient words.

\begin{lemma}[MUS lifting]
\label{lem:mus-lifting}
For every $m\ge4$ and $1\le i\le8$,
\begin{equation}
M_i^{(m+1)}=\operatorname{trim}\!\left(\mu(M_i^{(m)})\right).
\label{eq:mus-trim-recurrence}
\end{equation}
Moreover, for every $r\ge0$,
\begin{equation}
\Occ_{t_{m+r+1}}\!\left(M_i^{(m+1)}\right)
=2\,\Occ_{t_{m+r}}\!\left(M_i^{(m)}\right),
\label{eq:occ-doubling}
\end{equation}
where $2S=\{2p:p\in S\}$.
\end{lemma}

\begin{proof}
By Theorem~\ref{thm:kishi-mus}, write
$
M_i^{(m)}=xWy,
$
where $x,y\in\{a,b\}$ and $W$ is either $t_{m-3}$ or $\overline{t_{m-3}}$. Since
\[
\mu(x)=x\overline{x},\qquad \mu(y)=y\overline{y},
\]
we have
$
\mu(xWy)=x\overline{x}\,\mu(W)\,y\overline{y},
$
and therefore
$
\operatorname{trim}(\mu(xWy))=\overline{x}\,\mu(W)\,y.
$
The middle word $\mu(W)$ is $t_{m-2}$ or its complement, so the trimmed image is one of the eight MUSs of $t_{m+1}$. If the unique occurrence of $M_i^{(m)}$ starts at $p$ in $t_m$, then its morphic image starts at $2p-1$ and the trimmed factor starts at $2p$. Hence the order of the eight unique starts is preserved, proving~\eqref{eq:mus-trim-recurrence} with the same index $i$.

The same observation gives one inclusion in~\eqref{eq:occ-doubling} in every larger ambient word. If $xWy$ starts at $p$ in $t_{m+r}$, then $\mu(xWy)$ starts at $2p-1$ in $t_{m+r+1}$, so $\overline{x}\mu(W)y$ starts at $2p$. Thus
\[
2\,\Occ_{t_{m+r}}(M_i^{(m)})
\subseteq
\Occ_{t_{m+r+1}}(M_i^{(m+1)}).
\]

For the converse, suppose that $\overline{x}\mu(W)y$ starts at $q$ in $t_{m+r+1}$, and put $L=|W|$. Then $\mu(W)$ starts at $q+1$. Since $L\ge2$, Lemma~\ref{lem:synchronization} implies that $q+1$ is odd, so $q=2p$ for some $p$. The aligned occurrence of $\mu(W)$ therefore decodes to an occurrence of $W$ beginning at $p+1$ in $t_{m+r}$.

The symbol at position $2p$ is the second symbol of the image of the preceding preimage letter. Since that symbol is $\overline{x}$, the preceding preimage letter is $x$. Similarly, the final symbol $y$ occurs at the first position of the image of the following preimage letter, so that following letter is $y$. Hence $xWy=M_i^{(m)}$ starts at $p$, proving the reverse inclusion.

No boundary occurrence is lost. In the converse direction $q$ is even and hence at least two, so the required preceding image block exists; the final symbol exists by the assumed occurrence and lies at the first position of a canonical block. In the forward direction trimming removes one symbol from each end of an already valid morphic image, so the resulting occurrence remains inside the finite ambient word.
\end{proof}

\subsection{A concrete example of the lifting rule}
\label{subsec:lifting-example}

It is useful to see the recurrence once before using it symbolically. The first order-four MUS in our ordering is
$
M_1^{(4)}=bbab,
$
whose unique occurrence in $t_4$ starts at position $2$. Applying the morphism gives
$
\mu(bbab)=babaabba,
$
and trimming the first and last symbols yields
$
M_1^{(5)}=abaabb.
$
The start position doubles from $2$ to $4$. More importantly, the same doubling rule applies to \emph{every} occurrence when these factors are embedded in later words. Thus an occurrence of $bbab$ beginning at, say, $14$ in a larger ambient generation corresponds to an occurrence of $abaabb$ beginning at $28$ in the next generation.

This distinction between ``unique in the source word'' and ``possibly repeated in the ambient word'' is central. The lifting lemma does not merely transport the unique defining occurrence of a MUS; it transports the entire ambient occurrence set.

\subsection{Finite base data}
\label{subsec:base-data}

The lifting rule means that only a small amount of explicit data is needed. At order four the eight MUSs and their unique starts are fixed. To compute the coverages used in the stable characterization, we also need to know where these same order-four factors occur in $t_5$ and $t_6$.

\begin{table}[ht]
\centering
\caption{The MUSs of $t_4$, ordered by their unique occurrence, and the starting positions of the same factors in $t_4,t_5,t_6$.}
\label{tab:base-occurrences}
\begin{tabular}{@{}cclll@{}}
\toprule
$i$ & $M_i^{(4)}$ & $\Occ_{t_4}$ & $\Occ_{t_5}$ & $\Occ_{t_6}$ \\
\midrule
1 & $bbab$ & $2$  & $2,14,26$ & $2,14,26,42,50$ \\
2 & $baba$ & $3$  & $3,15,27$ & $3,15,27,43,51$ \\
3 & $abaa$ & $4$  & $4,16,28$ & $4,16,28,44,52$ \\
4 & $aabb$ & $6$  & $6,24$    & $6,24,30,40,54$ \\
5 & $bbaa$ & $8$  & $8,22$    & $8,22,32,38,56$ \\
6 & $aaba$ & $10$ & $10,18$   & $10,18,34,46,58$ \\
7 & $abab$ & $11$ & $11,19$   & $11,19,35,47,59$ \\
8 & $babb$ & $12$ & $12,20$   & $12,20,36,48,60$ \\
\bottomrule
\end{tabular}
\end{table}

The table should be read horizontally. For example, $M_1^{(4)}=bbab$ is unique in $t_4$ but already has three occurrences in $t_5$ and five in $t_6$. Those additional occurrences are precisely why an old MUS cannot be represented, in a later ambient word, only by its original interval.

The table is constant-size. Once it is known, Lemma~\ref{lem:mus-lifting} scales it to arbitrary order.

\subsection{From occurrence starts to coverage formulas}
\label{subsec:from-starts-to-coverages}

Repeated application of~\eqref{eq:occ-doubling} gives the exact start sets needed later. For $d\in\{0,1,2\}$ and $n\ge4+d$,
\begin{equation}
\Occ_{t_n}\!\left(M_i^{(n-d)}\right)
=
2^{\,n-4-d}\,
\Occ_{t_{4+d}}\!\left(M_i^{(4)}\right).
\label{eq:scaled-occurrences}
\end{equation}
Here multiplication of a set by a power of two means scaling every starting position by that factor.

The three values of $d$ correspond exactly to the three source generations that will appear in the stable witness system:
\[
M_i^{(n)},\qquad M_i^{(n-1)},\qquad M_i^{(n-2)}.
\]
Their lengths are
\[
2^{n-3}+2,\qquad 2^{n-4}+2,\qquad 2^{n-5}+2,
\]
respectively. Combining these lengths with the scaled start sets in~\eqref{eq:scaled-occurrences} immediately determines every component of their ambient coverage in $t_n$.

This is the point at which the problem changes character. Up to now we have studied factors and morphic alignment. From the next section onward, the relevant information can be represented almost entirely as unions of explicit integer intervals. The classification of the smallest attractors will then become a small transversal problem on those intervals.

\subsection{Relation to occurrence-preserving morphisms}
\label{subsec:relation-opm}

The occurrence identity above is deliberately specialized. Kishi et al.~\cite{Kishi2026} prove a general occurrence correspondence for ordinary morphic images under their occurrence-preserving framework. Lemma~\ref{lem:mus-lifting} is not intended as a generalization of that theory. Its role is narrower: the factors required by the present attractor problem are the actual consecutive-generation MUSs, and these are trimmed images rather than ordinary images.

The direct synchronization argument supplies exactly the additional information needed to handle the one-symbol deletion at both ends and to show that no ambient occurrence is created or lost. Once this is established, the entire later coverage analysis follows from the finite base data in Table~\ref{tab:base-occurrences}.

\paragraph{Roadmap.}
The next part of the extended version uses these transported coverages to prove the stable characterization. We will first display the current-, previous-, and second-previous-generation coverage intervals in a common scale, then explain geometrically why sixteen selected constraints force any four-position hitting set into one of two configurations. Only after that forcing argument will we return to the centered-boundary geometry to prove that every set in those two families is indeed an attractor.

\section{Stable characterization of the smallest attractors}
\label{sec:stable-characterization}

We now turn from occurrence transport to the classification problem itself.  The previous section reduced the relevant information to explicit coverage intervals.  The central observation of this section is that, once all intervals are written on the same scale, the proof becomes a small geometric hitting-set argument.

Throughout the section let $n\ge6$ and put
$
 h=2^{n-3},
$
so that $|t_n|=8h$ and $h\ge8$.  The scale $h$ is convenient because the four canonical boundaries of $t_n$ occur at
\[
 2h,\qquad 3h,\qquad 4h,\qquad 6h.
\]
At this resolution the final two families of smallest attractors are
\[
\FL_n=
\bigl\{\{2h+\varepsilon_1,3h+\varepsilon_2,4h+\varepsilon_3,6h+\varepsilon_4\}:
\varepsilon_1,\ldots,\varepsilon_4\in\{0,1\}\bigr\}
\]
and
\[
\FR_n=
\bigl\{\{2h+\varepsilon_1,4h+\varepsilon_2,5h+\varepsilon_3,6h+\varepsilon_4\}:
\varepsilon_1,\ldots,\varepsilon_4\in\{0,1\}\bigr\}.
\]
The purpose of the next subsections is to explain first why every four-position attractor is forced into one of these two shapes, and then why every such set is indeed an attractor.

\subsection{Three generations of coverage constraints}
\label{subsec:ABC-coverages}

For each $1\le i\le8$, define
\[
 A_i=\Cov_{t_n}(M_i^{(n)}),\qquad
 B_i=\Cov_{t_n}(M_i^{(n-1)}),\qquad
 C_i=\Cov_{t_n}(M_i^{(n-2)}).
\]
The superscript records the \emph{source generation} of the MUS.  Thus $M_i^{(n-1)}$ is not a MUS of the ambient word $t_n$; it is simply a factor of $t_n$, typically with several occurrences, and $B_i$ is the union of all those occurrence intervals.  The same convention applies to $C_i$.

Equation~\eqref{eq:scaled-occurrences} and Table~\ref{tab:base-occurrences} give the current-generation coverages
\begin{align*}
A_1&=[h,2h+1], & A_2&=[3h/2,5h/2+1],\\
A_3&=[2h,3h+1], & A_4&=[3h,4h+1],\\
A_5&=[4h,5h+1], & A_6&=[5h,6h+1],\\
A_7&=[11h/2,13h/2+1], & A_8&=[6h,7h+1].
\end{align*}
These are single intervals because the corresponding factors are MUSs of the ambient word $t_n$ and therefore occur once.

The proof also uses the following previous-generation coverages:
\begin{align*}
B_1&=[h/2,h+1]\cup[7h/2,4h+1]\cup[13h/2,7h+1],\\
B_2&=[3h/4,5h/4+1]\cup[15h/4,17h/4+1]\cup[27h/4,29h/4+1],\\
B_3&=[h,3h/2+1]\cup[4h,9h/2+1]\cup[7h,15h/2+1],\\
B_6&=[5h/2,3h+1]\cup[9h/2,5h+1],\\
B_7&=[11h/4,13h/4+1]\cup[19h/4,21h/4+1],\\
B_8&=[3h,7h/2+1]\cup[5h,11h/2+1],
\end{align*}
and two second-previous-generation coverages:
\begin{align*}
C_3={}&[h/2,3h/4+1]\cup[2h,9h/4+1]\cup[7h/2,15h/4+1]\\
&\cup[11h/2,23h/4+1]\cup[13h/2,27h/4+1],\\
C_6={}&[5h/4,3h/2+1]\cup[9h/4,5h/2+1]\cup[17h/4,9h/2+1]\\
&\cup[23h/4,6h+1]\cup[29h/4,15h/2+1].
\end{align*}

At first sight these formulas may look ad hoc.  Their role is in fact highly structured.  The $A_i$ determine broad regions in which the four attractor positions must lie.  The $B_i$ refine those regions and distinguish two global arrangements.  Finally, $C_3$ and $C_6$ fix the remaining endpoint choices.  The whole forcing argument therefore needs only sixteen factors.

Define
\[
\mathcal W_n=
\{M_i^{(n)}:1\le i\le8\}
\cup
\{M_i^{(n-1)}:i\in\{1,2,3,6,7,8\}\}
\cup
\{M_3^{(n-2)},M_6^{(n-2)}\}.
\]
Thus $|\mathcal W_n|=8+6+2=16$.

\subsection{Geometric forcing by sixteen witnesses}
\label{subsec:sixteen-witness}

Before giving the formal proof, it is useful to isolate its geometric backbone.  Four of the witness coverages,
\[
 A_2,\qquad B_7,\qquad B_2,\qquad A_7,
\]
are pairwise disjoint.  Hence a four-position set hitting all witnesses has no freedom to ``spend'' two positions in the same one of these regions: it must place exactly one position in each.  Once the four positions have been named according to those regions, the remaining coverage constraints can only be satisfied in two ways.  These two ways become the left and right families $\FL_n$ and $\FR_n$.

\begin{lemma}[Sixteen-witness lemma]
\label{lem:sixteen-witness}
Let $n\ge6$ and let $\Gamma\subseteq[1,2^n]$ have cardinality four.  If
\[
 \Gamma\cap\Cov_{t_n}(x)\ne\varnothing
 \qquad\text{for every }x\in\mathcal W_n,
\]
then
$
 \Gamma\in\FL_n\cup\FR_n.
$
\end{lemma}

\begin{proof}
Write $B_2^L,B_2^M,B_2^R$ for the three components of $B_2$, and $B_7^L,B_7^R$ for the two components of $B_7$.  Direct comparison of endpoints gives
\[
 B_2^L<A_2<B_7^L<B_2^M<B_7^R<A_7<B_2^R,
\]
where $X<Y$ means that every position of $X$ precedes every position of $Y$.  The smallest gap between consecutive sets in this display is $h/4-1>0$.  Hence $A_2,B_7,B_2,A_7$ are pairwise disjoint.

Since $\Gamma$ has exactly four positions and hits all four sets, every position of $\Gamma$ belongs to their union, one position to each set.  We may therefore name them uniquely as
\[
 p\in A_2,\qquad s\in B_7,\qquad r\in B_2,\qquad q\in A_7.
\]
The points $p$ and $q$ lie respectively before $A_4$ and after $A_5$.  If $r$ belonged to $B_2^L$ or $B_2^R$, then it would hit neither $A_4$ nor $A_5$.  The single point $s$ would then have to hit both $A_4$ and $A_5$, hence would have to belong to
$
 A_4\cap A_5=\{4h,4h+1\},
$
which is disjoint from $B_7$.  Therefore
$
 r\in B_2^M=[15h/4,17h/4+1].
$
We now distinguish the two components of $B_7$.

\smallskip
\noindent\emph{Left case: $s\in B_7^L$.}
Only $p$ can hit $A_1$, and only $r$ can hit $A_5$, so
\[
 p\le2h+1,\qquad r\ge4h.
\]
Only $q$ can hit either $A_6$ or $A_8$, and therefore it must hit both:
$
 q\in A_6\cap A_8=\{6h,6h+1\}.
$
Under the restrictions already obtained, only $r$ can hit $B_1$, and
$
 B_1\cap B_2^M\cap[4h,\infty)=\{4h,4h+1\}.
$
Thus $r\in\{4h,4h+1\}$.  Next, only $s$ can hit both $B_6$ and $B_8$, giving
$
 s\in B_7^L\cap B_6\cap B_8=\{3h,3h+1\}.
$
Finally, none of $q,r,s$ hits $C_3$, so $p$ must hit it.  Hence
$
 p\in A_2\cap(-\infty,2h+1]\cap C_3=\{2h,2h+1\}.
$
Therefore $\Gamma\in\FL_n$.

\smallskip
\noindent\emph{Right case: $s\in B_7^R$.}
Only $p$ can hit $A_1$ or $A_3$, so it must hit both:
$
 p\in A_1\cap A_3=\{2h,2h+1\}.
$
Only $r$ can hit $A_4$, and after this restriction only $r$ can also hit $B_3$.  Thus
$
 r\in B_2^M\cap A_4\cap B_3=\{4h,4h+1\}.
$
Only $q$ can hit $A_8$, so $q\ge6h$.  The coverages $B_6$ and $B_8$ must both be hit by $s$, yielding
$
 s\in B_7^R\cap B_6\cap B_8=\{5h,5h+1\}.
$
None of $p,r,s$ hits $C_6$, hence $q$ must.  Together with $q\in A_7$ and $q\ge6h$, this gives
$
 q\in A_7\cap[6h,\infty)\cap C_6=\{6h,6h+1\}.
$
Therefore $\Gamma\in\FR_n$.
\end{proof}

\paragraph{An internal lower bound.}
The four pairwise disjoint factor coverages $A_2,B_7,B_2,A_7$ already show that every attractor of $t_n$ has at least four positions when $n\ge6$.  Thus the stable regime contains its own lower-bound proof:
$
 \gamma(t_n)\ge4.
$
Combined with the sufficiency argument below, this independently recovers $\gamma(t_n)=4$ for all $n\ge6$.

\subsection{Why the two forced families are attractors}
\label{subsec:sufficiency-stable}

The witness lemma proves necessity: any four-position set satisfying the factor constraints must have one of two shapes.  To finish the classification we need sufficiency.  This is where the centered-occurrence property from Section~\ref{sec:preliminaries} becomes useful.

A boundary at $c$ means the adjacent pair $c\mid c+1$.  Recall that every factor of length at least two has an occurrence containing both endpoints of one of the four canonical boundaries
\[
 2h\mid2h+1,\qquad
 3h\mid3h+1,\qquad
 4h\mid4h+1,\qquad
 6h\mid6h+1.
\]
Therefore a set choosing one endpoint from each of these four boundaries automatically attracts every factor of length at least two.

\begin{lemma}[Paired boundaries]
\label{lem:paired-boundaries}
Every member of $\FL_n$ is a string attractor of $t_n$ for $n\ge6$.
\end{lemma}

\begin{proof}
Let $\Gamma\in\FL_n$.  By definition it contains exactly one endpoint of each of the four boundaries
\[
 2h\mid2h+1,\qquad 3h\mid3h+1,\qquad
 4h\mid4h+1,\qquad 6h\mid6h+1.
\]
Let $x$ be a factor of $t_n$ with $|x|\ge2$.  By the centered-occurrence property, some occurrence of $x$ contains both endpoints of one of these boundaries.  In particular it contains whichever endpoint was selected by $\Gamma$.  Thus every factor of length at least two is attracted.

It remains to check the one-letter factors.  Put $A=t_{n-3}$ and $B=\overline{A}$.  At the level of eight blocks of length $h$,
$
t_n=ABBABAAB.
$
The first letter of $A$ is always $a$; its last letter is $a$ when $n-3$ is even and $b$ when $n-3$ is odd, while the corresponding letters of $B$ are complementary.  Hence two of the four selected boundaries are homogeneous: when $n-3$ is even, the boundaries at $2h$ and $6h$ are respectively $bb$ and $aa$; when $n-3$ is odd, the boundaries at $3h$ and $4h$ are respectively $aa$ and $bb$.  Choosing either endpoint of each such boundary therefore guarantees at least one selected $a$-position and one selected $b$-position.  Thus $\Gamma$ attracts the one-letter factors as well.
\end{proof}

The second family follows from symmetry rather than from a separate coverage argument.

\begin{lemma}[Reflection]
\label{lem:reflection-family}
Every member of $\FR_n$ is a string attractor of $t_n$ for $n\ge6$.
\end{lemma}

\begin{proof}
A straightforward induction gives
\[
 t_n^R=
 \begin{cases}
 t_n,& n\text{ even},\\
 \overline{t_n},& n\text{ odd}.
 \end{cases}
\]
Complementing all letters preserves occurrence positions and hence preserves attractor positions.  Reflection of a word of length $N$ maps an attractor $\Gamma$ to
$
 \{N+1-p:p\in\Gamma\}.
$
Here $N=8h$, so the reflection of
$
 \{2h+\varepsilon_1,3h+\varepsilon_2,4h+\varepsilon_3,6h+\varepsilon_4\}\in\FL_n
$
is, after reordering,
\[
 \{2h+(1-\varepsilon_4),4h+(1-\varepsilon_3),5h+(1-\varepsilon_2),6h+(1-\varepsilon_1)\},
\]
which is an arbitrary member of $\FR_n$.  Lemma~\ref{lem:paired-boundaries} therefore implies the claim.
\end{proof}

\subsection{Complete classification and the number 32}
\label{subsec:complete-stable-classification}

We can now combine the forcing and sufficiency directions.

\begin{theorem}[Complete stable characterization]
\label{thm:stable-characterization}
For every $n\ge6$,
$
 \Acal(t_n)=\FL_n\cup\FR_n.
$
Consequently,
$
 |\Acal(t_n)|=32.
$
\end{theorem}

\begin{proof}
The disjoint coverages $A_2,B_7,B_2,A_7$ imply $\gamma(t_n)\ge4$, while Lemma~\ref{lem:paired-boundaries} supplies four-position attractors, so $\gamma(t_n)=4$.

Let $\Gamma\in\Acal(t_n)$.  Every factor in $\mathcal W_n$ is a factor of $t_n$, hence $\Gamma$ intersects all its ambient coverage sets.  Since $|\Gamma|=4$, Lemma~\ref{lem:sixteen-witness} yields
$
 \Gamma\in\FL_n\cup\FR_n.
$
Conversely, Lemmas~\ref{lem:paired-boundaries} and~\ref{lem:reflection-family} show that every member of either family is an attractor.

For $j\in\{2,3,4,5,6\}$, let $P_j=\{jh,jh+1\}$.  These pairs are mutually disjoint.  Every member of $\FL_n$ independently selects one position from each of $P_2,P_3,P_4,P_6$, so $|\FL_n|=2^4=16$.  Similarly $|\FR_n|=16$.  The two families are disjoint because every member of $\FL_n$ meets $P_3$, whereas no member of $\FR_n$ does.  Hence there are exactly $32$ smallest attractors.
\end{proof}

\subsection{Two constant-size certificates for four-position sets}
\label{subsec:constant-certificates}

The classification proof immediately gives a useful reformulation.  Once the cardinality is known to be four, testing all factors is unnecessary: the sixteen factors in $\mathcal W_n$ already form an exact certificate.

\begin{corollary}[Sixteen-factor certificate]
\label{cor:sixteen-certificate}
For $n\ge6$ and every $\Gamma\subseteq[1,2^n]$ with $|\Gamma|=4$, the following are equivalent:
\begin{enumerate}[label=(\roman*),leftmargin=*]
\item $\Gamma$ is a string attractor of $t_n$;
\item $\Gamma\cap\Cov_{t_n}(x)\ne\varnothing$ for every $x\in\mathcal W_n$.
\end{enumerate}
\end{corollary}

\begin{proof}
The forward implication is immediate from the definition of attractor.  Conversely, condition~(ii) and Lemma~\ref{lem:sixteen-witness} place $\Gamma$ in $\FL_n\cup\FR_n$, and Lemmas~\ref{lem:paired-boundaries} and~\ref{lem:reflection-family} show that every such set is an attractor.
\end{proof}

There is also a more symmetric, although slightly larger, certificate.  For $d\in\{1,2,3\}$ define
$
 \mathcal C_n^{[d]}=\bigcup_{j=0}^{d-1}\MUS(t_{n-j}),
$
where every factor is named in its source generation but its coverage is always taken in the ambient word $t_n$.

\begin{corollary}[Three-generation MUS certificate]
\label{cor:three-generation}
For $n\ge6$ and every four-position set $\Gamma\subseteq[1,2^n]$,
\[
 \Gamma\text{ is an attractor of }t_n
 \quad\Longleftrightarrow\quad
 \Gamma\cap\Cov_{t_n}(x)\ne\varnothing
 \text{ for every }x\in\mathcal C_n^{[3]}.
\]
The family $\mathcal C_n^{[3]}$ contains exactly $24$ factors.
\end{corollary}

\begin{proof}
The forward implication is immediate.  For the converse, note that $\mathcal W_n\subseteq\mathcal C_n^{[3]}$, so Corollary~\ref{cor:sixteen-certificate} applies.  The three generations contribute eight factors each.  Their lengths are
\[
 2^{n-3}+2,\qquad 2^{n-4}+2,\qquad 2^{n-5}+2,
\]
which are pairwise distinct for $n\ge6$.  Thus no factor is counted twice and the total is $24$.
\end{proof}

The distinction between the two certificates is useful.  The sixteen-factor family is the smallest witness system needed by the present proof, while the $24$-factor family has the conceptual advantage of consisting of three complete MUS generations.

\subsection{Why three generations are really needed}
\label{subsec:depth-optimality}

The preceding corollary naturally raises a question: could two consecutive MUS generations already suffice for four-position sets?  The answer is no.  The obstruction is explicit and persists at every order.

\begin{proposition}[Optimality of the three-generation depth]
\label{prop:depth-optimality}
For every $n\ge6$, depth three is the minimum $d\in\{1,2,3\}$ for which hitting the ambient coverages of every factor in $\mathcal C_n^{[d]}$ characterizes the four-position attractors of $t_n$.
\end{proposition}

\begin{proof}
Depth three is sufficient by Corollary~\ref{cor:three-generation}.  To show that depth two is insufficient, define
$
 \Gamma_n^\star=\{3h/2,3h,4h,6h\}.
$
The displayed formulas for the $A_i$ show that all current-generation MUS coverages are hit: $3h/2$ hits $A_1,A_2$; $3h$ hits $A_3,A_4$; $4h$ hits $A_5$; and $6h$ hits $A_6,A_7,A_8$.

For completeness, the two previous-generation coverages not displayed earlier are
$
 B_4=[3h/2,2h+1]\cup[6h,13h/2+1],
$
and
$
 B_5=[2h,5h/2+1]\cup[11h/2,6h+1].
$
Together with the six previous $B_i$ formulas, they show that $\Gamma_n^\star$ hits every ambient coverage arising from $\MUS(t_n)\cup\MUS(t_{n-1})$.

However, $\Gamma_n^\star$ misses $C_3=\Cov_{t_n}(M_3^{(n-2)})$.  Indeed,
\begin{align*}
3h/2&\in(3h/4+1,2h),\\
3h&\in(9h/4+1,7h/2),\\
4h&\in(15h/4+1,11h/2),\\
6h&\in(23h/4+1,13h/2),
\end{align*}
and these four intervals are gaps between consecutive components of $C_3$.  Hence $\Gamma_n^\star$ is not an attractor.  Depth two is therefore insufficient, and depth one is insufficient a fortiori.
\end{proof}

\begin{remark}[Scope of the optimality statement]
The proposition concerns only the nested hierarchy of certificates obtained from one, two, or three \emph{consecutive complete MUS generations}.  It does not claim that every possible factor certificate must contain a factor originating in $t_{n-2}$, nor does it rule out certificates mixing factors from the first two generations with other non-MUS witnesses.
\end{remark}

\subsection{Why the cardinality restriction matters}
\label{subsec:cardinality-restriction}

Both constant-size certificates above rely essentially on the hypothesis $|\Gamma|=4$.  Without that restriction, hitting three generations of MUS coverages does not characterize attractors.

For example, in $t_6$ the five-position set
$
 \Gamma_0=\{4,12,24,34,48\}
$
hits the ambient coverages of all $24$ factors in
$
 \MUS(t_6)\cup\MUS(t_5)\cup\MUS(t_4),
$
but misses $\Cov_{t_6}(bb)$.  It therefore passes the three-generation MUS test while failing to be an attractor.

This example explains why the next part of the extended version requires a different viewpoint.  To characterize attractors of \emph{arbitrary} cardinality, one must identify the complete antichain of inclusion-minimal factor coverages rather than rely only on a constant number of recent MUS generations.

\section{Canonical minimal factor coverages}
\label{sec:canonical-coverages}

The constant-size certificates of the previous section exploit a very strong hypothesis: the candidate set has exactly four positions.  Under that restriction, a relatively small collection of MUS coverages can force the entire geometry of a smallest attractor.  For arbitrary position sets, however, the example at the end of Section~\ref{subsec:cardinality-restriction} shows that three recent MUS generations are no longer enough.

The right object is therefore not a fixed collection of factors but the partial order of \emph{all} factor coverages under inclusion.  This section identifies its minimal elements completely.  Besides giving an exact certificate for attractors of arbitrary cardinality, the result reveals a second recursive structure hidden in the Thue--Morse word: every essential long coverage constraint descends, under the morphism, to one of the eight order-four minimal unique substrings.

For $n\ge5$, define
$
 \mathcal D_n=
 \{\Cov_{t_n}(aa),\Cov_{t_n}(bb)\}
 \cup
 \bigcup_{m=4}^{n}
 \{\Cov_{t_n}(M):M\in\MUS(t_m)\}.
$
As throughout the paper, a factor $M\in\MUS(t_m)$ is named in its \emph{source generation} $t_m$, while its coverage is always taken in the ambient word $t_n$.  The purpose of the section is to prove
$
 \Hmin(t_n)=\mathcal D_n.
$

It is useful to separate the proof into two logically different tasks.
\begin{enumerate}[label=(\arabic*),leftmargin=*]
\item \emph{Exhaustion:} show that every inclusion-minimal factor coverage belongs to $\mathcal D_n$.
\item \emph{Irredundancy:} show that no two members of $\mathcal D_n$ contain one another.
\end{enumerate}
The first part is governed by desubstitution and frequency criticality; the second is governed by the geometry and number of occurrences.

\subsection{From minimal coverages to critical factors}
\label{subsec:critical-representatives}

Suppose that two factors have the same coverage.  The shorter one is preferable when studying minimality, because it places a stronger local requirement on the occurrence structure.  This simple observation leads to the following notion.

For a factor $x$ of length at least two, let $x^-$ and $x^+$ denote respectively its prefix and suffix of length $|x|-1$.

\begin{definition}[Two-sided frequency criticality]
A factor $x$ of a finite word $w$ is \emph{two-sided frequency-critical} if
\[
 |\Occ_w(x^-)|>|\Occ_w(x)|
 \qquad\text{and}\qquad
 |\Occ_w(x^+)|>|\Occ_w(x)|.
\]
Thus deleting either the first or the last symbol creates at least one additional occurrence.
\end{definition}

The definition is deliberately phrased in terms of frequencies rather than uniqueness.  A MUS is an extreme example: it occurs once, while its longest proper prefix and suffix occur at least twice.  In the present section we must also deal with factors that are repeated in the ambient word, so ordinary MUS terminology is no longer sufficient.

\begin{lemma}[Critical representatives]
\label{lem:critical-representatives}
Let $C\in\Hmin(w)$ and suppose that $C$ has a factor representative of length at least two.  If $x$ is a shortest such representative, then $x$ is two-sided frequency-critical.
\end{lemma}

\begin{proof}
Assume first that
$
 |\Occ_w(x^-)|=|\Occ_w(x)|.
$
Every occurrence of $x$ induces an occurrence of $x^-$ at the same starting position.  Equality of the two cardinalities therefore implies that \emph{every} occurrence of $x^-$ extends one position to the right to an occurrence of $x$.  Consequently,
$
 \Cov_w(x^-)\subseteq\Cov_w(x)=C.
$
Since $C$ is inclusion-minimal in $\mathcal H(w)$, the inclusion must be an equality.  But then $x^-$ is a shorter factor representative of $C$, contradicting the choice of $x$.

Hence $|\Occ_w(x^-)|>|\Occ_w(x)|$.  The argument for $x^+$ is symmetric after indexing occurrences by their final positions.
\end{proof}

The lemma is elementary, but conceptually important: it converts a statement about the poset of coverage sets into a local statement about how occurrence frequencies change when one symbol is removed.

\subsection{Unique phase and critical descent}
\label{subsec:critical-descent}

We next exploit a special feature of the Thue--Morse morphism.  A sufficiently long factor determines the parity class of all its occurrences.  Informally, after reading four symbols one already knows where the canonical length-two image blocks begin.  This makes desubstitution unambiguous.

For $\lambda,\rho\in\{0,1\}$, write
$
 \operatorname{trim}_{\lambda,\rho}(z)
   =z[1+\lambda\, ..\, |z|-\rho].
$
Thus $\lambda=1$ removes the first symbol and $\rho=1$ removes the last one.

\begin{lemma}[Factor synchronization and critical descent]
\label{lem:factor-critical-descent}
Every factor $x$ of $t_n$ with $|x|\ge4$ has a unique morphic phase.  More precisely, there are uniquely determined $\lambda,\rho\in\{0,1\}$ and a factor $v$ of $t_{n-1}$ such that
$
 x=\operatorname{trim}_{\lambda,\rho}(\mu(v)),
$
and the occurrences of $x$ in $t_n$ correspond bijectively to the occurrences of $v$ in $t_{n-1}$, with the appropriate doubled-start offset.

If, in addition, $x$ is two-sided frequency-critical and $|x|\ge5$, then
$
 \lambda=\rho=1,
 \qquad
 x=\operatorname{trim}(\mu(v)),
$
and $v$ is two-sided frequency-critical in $t_{n-1}$.
\end{lemma}

\begin{proof}
The ten length-four factors of the Thue--Morse word split into two disjoint phase classes.  With one-based positions,
\[
 abab,\ abba,\ baab,\ baba
\]
start only at odd positions, whereas
\[
 aaba,\ aabb,\ abaa,\ babb,\ bbaa,\ bbab
\]
start only at even positions.  This can be read directly from two consecutive canonical images $\mu(a)=ab$ and $\mu(b)=ba$ and from the factors crossing an image boundary.

The length-four prefix of $x$ therefore fixes the phase of every occurrence of $x$.  Completing at most one symbol on each side aligns the factor with canonical image blocks.  Whenever a completion symbol is needed, it is forced because the two symbols of every image block are complementary.  Decoding the aligned blocks gives a unique factor $v$ and the claimed bijection between occurrence starts.

Now assume that $x$ is critical.  If $\lambda=0$, then $x$ begins at an image boundary.  Its suffix $x^+$ has length at least four and starts in the opposite phase.  Since that phase is unique, every occurrence of $x^+$ has a uniquely determined preceding symbol, and that symbol reconstructs $x$.  Hence
$
 |\Occ_{t_n}(x^+)|=|\Occ_{t_n}(x)|,
$
contradicting criticality.  Thus $\lambda=1$.  The symmetric argument applied to the prefix $x^-$ gives $\rho=1$.

Finally, when $x=\operatorname{trim}(\mu(v))$, the same phase correspondence identifies occurrences of $x^-$ and $x^+$ with occurrences of $v^-$ and $v^+$, respectively.  Therefore the two strict frequency inequalities for $x$ transfer directly to $v$.
\end{proof}

The key consequence is recursive: a long essential factor does not merely desubstitute; it desubstitutes to another factor satisfying the same criticality condition.  We can therefore descend repeatedly until length four.

\subsection{The critical factors at the base level}
\label{subsec:critical-base}

The length-four factors of $t_4$ are
\[
 aaba,\ aabb,\ abaa,\ abab,\ abba,\ baab,\ baba,\ babb,\ bbaa,\ bbab.
\]
Exactly eight of them are two-sided critical:
\[
 aaba,\ aabb,\ abaa,\ abab,\ baba,\ babb,\ bbaa,\ bbab.
\]
These are precisely the eight members of $\MUS(t_4)$ from Table~\ref{tab:base-occurrences}.  The two remaining factors, $abba$ and $baab$, fail criticality on both sides.

For later use we make explicit why this finite classification remains valid in every ambient order.  For a factor $u$ of $t_r$, let $R_r(u)$ and $L_r(u)$ denote the sets of right and left extensions of $u$ in $t_r$, including a boundary marker when $u$ is a suffix or a prefix.  Then a factor $x$ is critical precisely when the prefix $x^-$ has an alternative right extension and the suffix $x^+$ has an alternative left extension.

The length-five factors already present in $t_4$ are
\begin{align*}
 &aabab,\ aabba,\ abaab,\ ababb,\ abbaa,\ abbab,\\
 &baaba,\ baabb,\ babaa,\ babba,\ bbaab,\ bbaba.
\end{align*}
This set is closed under complementation.  Moreover, $t_r$ begins with $abba$, while its last four symbols are $abba$ for even $r$ and $baab$ for odd $r$.  Hence the four length-five factors crossing the central join of $t_{r+1}=t_r\overline{t_r}$ are
\[
 \begin{array}{ll}
 r\text{ even}:& abbab,\ bbaba,\ babaa,\ abaab,\\
 r\text{ odd}:& baabb,\ aabba,\ abbaa,\ bbaab,
 \end{array}
\]
all of which already occur in $t_4$.  Inductively, the internal one-letter extension sets of factors of length at most four therefore stabilize from order four onward.  Including the two word boundaries yields the exhaustive pattern
\begin{center}
\begin{tabular}{lcc}
\toprule
Factor class & Alternative right & Alternative left\\
\midrule
$aab,bba$ & no & yes\\
$abb,baa$ & yes & no\\
$abba,baab$ & no & no\\
all other factors of lengths $2$--$4$ & yes & yes\\
\bottomrule
\end{tabular}
\end{center}
for every ambient $t_r$, $r\ge4$.

Thus the two- and three-letter critical factors are
\[
 aa,\ bb,\ ab,\ ba,\ aba,\ bab,
\]
and the eight critical length-four factors are exactly $\MUS(t_4)$.

Combining this finite base with Lemma~\ref{lem:factor-critical-descent} yields the classification of all possible minimal-coverage representatives.

\begin{lemma}[Classification of possible minimal representatives]
\label{lem:possible-minimal-representatives}
For every $n\ge5$,
$
 \Hmin(t_n)\subseteq\mathcal D_n.
$
\end{lemma}

\begin{proof}
Neither one-letter coverage is minimal, because
\[
 \Cov_{t_n}(aa)\subsetneq\Cov_{t_n}(a),
 \qquad
 \Cov_{t_n}(bb)\subsetneq\Cov_{t_n}(b).
\]
Hence every inclusion-minimal coverage has a representative of length at least two.  Choose a shortest one.  By Lemma~\ref{lem:critical-representatives}, it is critical.

If its length is at least five, repeated application of Lemma~\ref{lem:factor-critical-descent} eventually reaches one of the eight critical length-four factors.  Reversing the descent and using the trimmed MUS recurrence of Lemma~\ref{lem:mus-lifting} shows that the original factor is $M_i^{(m)}\in\MUS(t_m)$ for some $4\le m\le n$.

It remains to remove the four short critical factors $ab,ba,aba,bab$.  The internal repetitions
\[
 abab=ab\,ab,
 \qquad
 baba=ba\,ba
\]
give
\[
 \Cov_{t_n}(abab)\subseteq\Cov_{t_n}(ab),
 \qquad
 \Cov_{t_n}(baba)\subseteq\Cov_{t_n}(ba).
\]
Similarly,
\[
 abaaba=aba\,aba,
 \qquad
 babbab=bab\,bab
\]
give
\[
 \Cov_{t_n}(abaaba)\subseteq\Cov_{t_n}(aba),
 \qquad
 \Cov_{t_n}(babbab)\subseteq\Cov_{t_n}(bab).
\]
The four factors on the left are MUSs of $t_4$ or $t_5$, and hence their coverages belong to $\mathcal D_n$.  Therefore none of $\Cov(ab),\Cov(ba),\Cov(aba),\Cov(bab)$ can be inclusion-minimal.  The only remaining possibilities are exactly the members of $\mathcal D_n$.
\end{proof}

At this point we know that the candidate family is exhaustive.  What remains is to show that none of its members is redundant.

\subsection{Occurrence geometry across all generations}
\label{subsec:all-generation-geometry}

The antichain proof uses two pieces of quantitative information: the spacing between distinct occurrences and the exact number of occurrences.  Both ultimately come from the same recursive transport law established in Section~\ref{sec:mus-transport}.

Let $M_i^{(m)}$ be the $i$th MUS of $t_m$ and put
\[
 s=2^{m-4},
 \qquad
 q=n-m+4.
\]
Iterating Lemma~\ref{lem:mus-lifting} gives
$
 \Occ_{t_n}(M_i^{(m)})
   =s\,\Occ_{t_q}(M_i^{(4)}),
$
where $sS=\{sp:p\in S\}$, and
$
 |M_i^{(m)}|=2s+2.
$

To count the base-factor occurrences, define
$
 c_i(q)=|\Occ_{t_q}(M_i^{(4)})|.
$
Complementation pairs the eight factors as
\[
(1,6),\qquad(2,7),\qquad(3,8),\qquad(4,5).
\]
Write $\bar i$ for the index paired with $i$.  Since
$
 t_{q+1}=t_q\overline{t_q},
$
all occurrences wholly contained in the two halves contribute
$
 c_i(q)+c_{\bar i}(q).
$
Only occurrences crossing the central join have to be added separately.  Inspecting the unique phase at that join gives
$
 c_i(q+1)=c_i(q)+c_{\bar i}(q)+\delta_i(q),
$
where the crossing correction is determined by parity: when $q$ is even, exactly the factors $M_1^{(4)},M_2^{(4)},M_3^{(4)}$ acquire one crossing occurrence; when $q$ is odd, the relevant crossing length-four factors are $M_4^{(4)}$, $abba$, and $M_5^{(4)}$.  The factor $abba$ does not belong to the MUS octuple, so only indices $4$ and $5$ receive the corresponding contribution.

Starting from $c_i(4)=1$ and using the complement pairs, a short induction gives the closed form
\[
 c_i(q)=
 \begin{cases}
 \dfrac{2^{q-2}-1}{3},
   & q\text{ even},\\[1.1ex]
 \dfrac{2^{q-2}+1}{3},
   & q\text{ odd and }i\in\{1,2,3\},\\[1.1ex]
 \dfrac{2^{q-2}-2}{3},
   & q\text{ odd and }i\in\{4,5,6,7,8\}.
 \end{cases}
\]

The second geometric fact is separation.  At order four, two distinct occurrences of the same $M_i^{(4)}$ start at least six positions apart.  This can be checked in $t_5$: every factor that could witness a smaller distance has length at most nine, and the Thue--Morse factor set of that length is already present there.  Synchronization then preserves the obstruction in all later orders.  Scaling the starts by $s$ yields the following summary.

\begin{lemma}[MUS coverage geometry]
\label{lem:mus-coverage-geometry-extended}
Let $4\le m\le n$, put $s=2^{m-4}$ and $q=n-m+4$.  Then
\[
 \Occ_{t_n}(M_i^{(m)})
   =s\,\Occ_{t_q}(M_i^{(4)}),
 \qquad
 |M_i^{(m)}|=2s+2,
\]
distinct occurrence starts differ by at least $6s$, and their number is $c_i(q)$ as given above.
\end{lemma}

A useful picture follows immediately.  The coverage of a generation-$m$ MUS is a disjoint union of intervals of length $2s+2$, while consecutive components are separated by gaps of at least
$
 6s-(2s+2)=4s-2.
$
Thus increasing the source generation makes each individual component longer but reduces the number of components.  The antichain proof is precisely a comparison between these two effects.

\subsection{Why the MUS-derived coverages are incomparable}
\label{subsec:mus-antichain}

We first compare two coverages arising from MUSs.

Within a fixed generation $m$, all components have the same length.  If
$
 \Cov_{t_n}(M_i^{(m)})
   \subseteq
 \Cov_{t_n}(M_j^{(m)}),
$
then every component of the left-hand side must coincide with a component of the right-hand side.  In particular the two distinct equal-length factors would occur at the same starting position, which is impossible.  Hence distinct MUS coverages from the same generation are incomparable.

Now compare generations $m$ and $m+d$, where $d\ge1$.  A component from the later source generation has length
$
 2^{d+1}s+2,
$
strictly larger than the earlier length $2s+2$.  Therefore a later-generation component cannot be contained in the earlier-generation coverage, whose components are separated by positive gaps.  This rules out one direction of inclusion.

For the reverse direction we count how many earlier components can fit inside a later component.  Their starts are separated by at least $6s$, while the available range for their starts inside an interval of length $2^{d+1}s+2$ is
$
 2s(2^d-1).
$
Hence one later component can contain at most
$
 K_d
 =1+\left\lfloor\frac{2s(2^d-1)}{6s}\right\rfloor
 =\left\lceil\frac{2^d}{3}\right\rceil
$
earlier components.

Put
$
 c_{\min}(r)=\min_i c_i(r),
 \qquad
 c_{\max}(r)=\max_i c_i(r).
$
The closed formulas imply
$
 c_{\min}(q+d)>K_d\,c_{\max}(q).
$
One way to see the strict inequality is to use
\[
 c_{\min}(r)\ge\frac{2^{r-2}-2}{3},
 \qquad
 c_{\max}(r)\le\frac{2^{r-2}+1}{3},
 \qquad
 K_d\le\frac{2^d+2}{3}.
\]
These estimates settle every case except $(q,d)=(4,1)$; for that boundary case the exact values are
\[
 c_{\min}(5)=2,
 \qquad
 K_1=1,
 \qquad
 c_{\max}(4)=1.
\]
Thus the earlier-generation coverage has more components than all later-generation components could collectively contain.  The reverse inclusion is impossible as well.

We have therefore proved that all MUS-derived coverages appearing in $\mathcal D_n$ form an antichain.

\subsection{The two homogeneous-pair constraints}
\label{subsec:aa-bb-antichain}

It remains to incorporate $\Cov_{t_n}(aa)$ and $\Cov_{t_n}(bb)$.  These two constraints look different from the MUS hierarchy, but their incomparability is especially simple.

First, no MUS coverage can be contained in either homogeneous-pair coverage.  Every Thue--Morse MUS contains both letters, whereas $\Cov_{t_n}(aa)$ contains only positions occupied by $a$ and $\Cov_{t_n}(bb)$ contains only positions occupied by $b$.

Conversely, fixed homogeneous pairs near the beginning of the word witness the opposite non-inclusions.  For every $n\ge5$,
\[
 t_n[2..3]=bb,
 \qquad
 t_n[6..7]=aa.
\]
For source generation $m\ge6$, the first occurrence of every $M_i^{(m)}$ starts at or after position $8$, so its coverage misses both of these fixed pairs.  The two low generations $m=4,5$ are checked directly from Table~\ref{tab:base-occurrences}; when a later $aa$ witness is needed one can use the fixed occurrences at $[10,11]$ or $[18,19]$.

Finally,
$
 \Cov_{t_n}(aa)\cap\Cov_{t_n}(bb)=\varnothing,
$
because a position cannot simultaneously carry $a$ and $b$.  Thus the two homogeneous-pair coverages are incomparable with one another and with every MUS-derived coverage.

\begin{lemma}[Canonical antichain]
\label{lem:canonical-antichain-extended}
For every $n\ge5$, the members of $\mathcal D_n$ are pairwise incomparable under inclusion.
\end{lemma}

\subsection{The canonical minimal-coverage theorem}
\label{subsec:minimal-coverage-theorem}

We can now combine exhaustion and irredundancy.  The proof is short because all of the work has been isolated in the preceding lemmas.

\begin{theorem}[Minimal factor coverages of Thue--Morse words]
\label{thm:minimal-coverages-extended}
For every $n\ge5$,
\[
 \Hmin(t_n)=
 \{\Cov_{t_n}(aa),\Cov_{t_n}(bb)\}
 \cup
 \bigcup_{m=4}^{n}
 \{\Cov_{t_n}(M):M\in\MUS(t_m)\}.
\]
All displayed coverages are distinct.  Consequently,
$
 |\Hmin(t_n)|=2+8(n-3)=8n-22.
$
\end{theorem}

\begin{proof}
Lemma~\ref{lem:possible-minimal-representatives} gives
$
 \Hmin(t_n)\subseteq\mathcal D_n.
$
Take any $D\in\mathcal D_n$.  Since $\mathcal H(t_n)$ is finite, $D$ contains some inclusion-minimal member $C\in\Hmin(t_n)$.  By the inclusion above, $C\in\mathcal D_n$.  Lemma~\ref{lem:canonical-antichain-extended} says that $\mathcal D_n$ is an antichain, so $C\subseteq D$ forces $C=D$.  Hence every member of $\mathcal D_n$ is inclusion-minimal and
$
 \Hmin(t_n)=\mathcal D_n.
$
There are eight MUSs in each generation $m\ge4$, and the antichain property implies that all displayed coverages are distinct.  The count is therefore
$
 2+8(n-3)=8n-22.
$
\end{proof}

The theorem explains the linear sequence of minimal-coverage counts visible in the finite computations:
\[
 18,\ 26,\ 34,\ 42,\ 50,\ldots
\]
for $n=5,6,7,8,9,\ldots$.  More importantly, it provides an exact structural reason for that sequence: every new Thue--Morse generation contributes exactly eight new incomparable MUS-derived constraints.

\begin{corollary}[Exact certificate for arbitrary attractors]
\label{cor:arbitrary-attractors-extended}
Let $n\ge5$ and let $\Gamma\subseteq[1,2^n]$ be arbitrary.  Then $\Gamma$ is a string attractor of $t_n$ if and only if it intersects
\[
 \Cov_{t_n}(aa),\qquad \Cov_{t_n}(bb),
\]
and every coverage
\[
 \Cov_{t_n}(M),
 \qquad
 M\in\MUS(t_m),\quad 4\le m\le n.
\]
\end{corollary}

This corollary should be contrasted with Corollary~\ref{cor:three-generation}.  For arbitrary attractors the exact canonical system has $8n-22$ constraints; once the cardinality is fixed to the optimum value four, three recent MUS generations already suffice, and the proof itself uses only sixteen witnesses.

\subsection{Why the certificate is irredundant}
\label{subsec:canonical-irredundancy}

The word ``canonical'' can be made precise at the level of coverage sets.  The following observation is completely general and does not use the Thue--Morse structure.

\begin{lemma}[Canonical irredundancy]
\label{lem:canonical-irredundancy-extended}
Let $w$ be a finite word and let $\mathcal F\subseteq\mathcal H(w)$ satisfy
\[
 \Gamma\text{ is an attractor of }w
 \quad\Longleftrightarrow\quad
 \Gamma\cap C\ne\varnothing
 \text{ for every }C\in\mathcal F
\]
for every position set $\Gamma$.  Then
$
 \Hmin(w)\subseteq\mathcal F.
$
\end{lemma}

\begin{proof}
Take $C\in\Hmin(w)$ and suppose that $C\notin\mathcal F$.  For every $D\in\mathcal F$, we must have $D\nsubseteq C$; otherwise $D\subseteq C$ and minimality of $C$ would imply $D=C$.  Choose one position
$
 p_D\in D\setminus C
$
for every $D\in\mathcal F$, and put
$
 \Gamma=\{p_D:D\in\mathcal F\}.
$
Then $\Gamma$ hits every member of $\mathcal F$ but misses $C$.  Hence $\Gamma$ is not an attractor, contradicting the assumed equivalence.
\end{proof}

Combining this lemma with Theorem~\ref{thm:minimal-coverages-extended} shows that the $8n-22$ constraints are not merely one convenient exact certificate: every exact certificate consisting of factor coverages must contain all of them.  Thus they form the unique minimum exact certificate \emph{at the level of coverage sets}.  The qualification is important because two different factor names could, in principle, induce the same coverage.

\subsection{The resulting hierarchy of certificates}
\label{subsec:certificate-hierarchy}

The two main parts of the paper can now be summarized by a simple hierarchy:
\begin{center}
\fbox{all attractors}
$\Longleftrightarrow$
\fbox{$8n-22$ canonical minimal coverages}
\end{center}
while under the minimum-cardinality restriction,
\begin{center}
\fbox{$|\Gamma|=4$}
$\Longrightarrow$
\fbox{24 MUS constraints}
$\Longrightarrow$
\fbox{16 forcing witnesses}.
\end{center}

The first level is canonical and irredundant but grows linearly with $n$.  The second and third levels are constant-size because the cardinality constraint itself supplies additional combinatorial information.  This distinction is useful beyond the particular formulas: it separates the intrinsic complexity of the attractor property from the much more rigid geometry of optimum solutions.


\section{The exceptional orders four and five}
\label{sec:exceptional-orders}

The stable theorem starts at order six.  This is not merely a technical threshold introduced by the proof.  At order five most of the stable geometry is already visible, but one additional family of smallest attractors survives; at order four the solution set is more irregular and is better described by a compact system of constraints than by a short parametric formula.

These two small orders are useful for two reasons.  First, they complete the classification for every $n\ge4$, the range in which $\gamma(t_n)=4$.  Second, they show where the stable geometry comes from: the two large families are already present at $n=5$, while the extra solutions disappear as soon as the recursive separation between the relevant coverage intervals becomes large enough.

\subsection{Order five: the stable families plus one exceptional bridge}
\label{subsec:t5-exceptional}

For $n=5$ we have $h=2^{5-3}=4$.  Substituting this value into the two stable formulas gives two $16$-element families $\FL_5$ and $\FR_5$.  They are not yet exhaustive.  Define
\[
 \mathcal E_5=
 \bigl\{\{8+\alpha,14+\eta,18+\eta,24+\delta\}:
       \alpha,\eta,\delta\in\{0,1\}\bigr\}.
\]
The same offset $\eta$ is used in the two central positions.  Thus $\mathcal E_5$ contains eight sets rather than sixteen.

\begin{proposition}[The exceptional order five]
\label{prop:t5-extended}
The smallest attractors of $t_5$ are exactly
$
 \Acal(t_5)=\FL_5\cup\mathcal E_5\cup\FR_5.
$
The three families are pairwise disjoint.  Consequently,
$
 |\Acal(t_5)|=16+8+16=40.
$
\end{proposition}

There is a useful geometric way to read this statement.  Every optimum solution contains one position in the left outer pair
$
 \{8,9\}
$
and one position in the right outer pair
$
 \{24,25\}.
$
The remaining two positions lie in the central part of the word and can occur in exactly three ways:
\[
 \{12,13\}\times\{16,17\},
 \qquad
 \{16,17\}\times\{20,21\},
\]
or as one of the two correlated pairs
\[
 (14,18),\qquad(15,19).
\]
The first and second rectangles give $\FL_5$ and $\FR_5$, respectively.  The two correlated pairs give $\mathcal E_5$.

This makes the transition to the stable regime particularly transparent.  From order six onward the central intervals separate sufficiently to eliminate the correlated bridge represented by $\mathcal E_5$.  What remains are precisely the two reflected families already visible at order five.

The proof of Proposition~\ref{prop:t5-extended} is finite but still structural.  The inclusion-minimal coverage system of $t_5$ consists of eighteen masks: the eight MUS-derived coverages from generation five, the eight corresponding coverages from generation four, and the two homogeneous-pair coverages $\Cov_{t_5}(aa)$ and $\Cov_{t_5}(bb)$.  Hitting these eighteen masks forces the two outer choices and one of the ten allowed central pairs above.  Conversely every set described by those choices hits all eighteen masks.  The complete table is recorded in Appendix~\ref{app:finite-tables}.

\subsection{Order four: a finite but genuinely less regular geometry}
\label{subsec:t4-exceptional}

The word
$
 t_4=abbabaabbaababba
$
has length sixteen.  Here a short family description analogous to $\FL_n$ and $\FR_n$ would hide rather than clarify the geometry.  A compact Boolean characterization is more informative.

Every smallest attractor has one position in the left region $[2,5]$, two positions in the central region $[6,11]$, and one position in the right region $[12,15]$.  Write it as
$
 \Gamma=\{\ell,u,v,r\},
$
with
\[
 2\le\ell\le5,
 \qquad
 6\le u<v\le11,
 \qquad
 12\le r\le15.
\]
The remaining restrictions are exactly the following ten clauses.

\begin{proposition}[The exceptional order four]
\label{prop:t4-extended}
A four-position set $\Gamma$ is a smallest attractor of $t_4$ if and only if it has the form above and satisfies
\begin{align*}
 &\ell\ge3\ \lor\ u=6,
 &&\ell\ge4\ \lor\ u\le7,
 &&u\le9,\\
 &v\ge8,
 &&v\ge10\ \lor\ r\le13,
 &&v=11\ \lor\ r\le14,\\
 &\ell\ge3\ \lor\ r\le14,
 &&u\le7\ \lor\ v\ge10,\\
 &\ell\le3\ \lor\ \{u,v\}\cap\{8,9\}\ne\varnothing\ \lor\ r\ge14,\\
 &\ell\ge4\ \lor\ u=6\ \lor\ v=11\ \lor\ r\le13.
\end{align*}
Exactly $87$ quadruples satisfy these conditions.
\end{proposition}

Although the displayed formula is less visually appealing than the stable description, its origin is simple.  The twelve inclusion-minimal coverage constraints of $t_4$ force the coarse decomposition into one left, two central, and one right position.  Expanding the twelve intersection conditions and deleting logical redundancies gives precisely the ten clauses above.  Thus the proposition is not an arbitrary Boolean encoding of an exhaustive search: it is a compact form of the complete minimal-coverage system.

The number $87$ can then be recovered without testing all $\binom{16}{4}=1820$ subsets one by one.  Conditioning on the leftmost position $\ell=2,3,4,5$, the numbers of admissible choices are respectively
\[
 10,\qquad21,\qquad28,\qquad28,
\]
and their sum is $87$.  The exhaustive verifier described later independently checks that these are exactly the four-subsets hitting all factor coverages.

\subsection{Small-order certificates and the complete enumeration}
\label{subsec:exceptional-certificates}

The exceptional cases also have concise MUS formulations.  For arbitrary position sets $\Gamma$,
\[
 \Gamma\text{ attracts }t_5
 \quad\Longleftrightarrow\quad
 \Gamma\text{ hits every MUS of }t_5,t_4,t_3,
\]
and
\[
 \Gamma\text{ attracts }t_4
 \quad\Longleftrightarrow\quad
 \Gamma\text{ hits every MUS of }t_4,t_3.
\]
Unlike the stable three-generation statement of Corollary~\ref{cor:three-generation}, these two equivalences do not require $|\Gamma|=4$.  They follow from the explicit minimal-coverage antichains in Appendix~\ref{app:finite-tables}.

Combining the two exceptional orders with Theorem~\ref{thm:stable-characterization} gives the complete enumeration
\[
 |\Acal(t_n)|=
 \begin{cases}
 87,&n=4,\\
 40,&n=5,\\
 32,&n\ge6.
 \end{cases}
\]
For completeness, the preceding four orders satisfy
$
(\gamma(t_0),\gamma(t_1),\gamma(t_2),\gamma(t_3))=(1,2,2,3)
$
and have respectively
\[
 1,\quad1,\quad2,\quad12
\]
smallest attractors.  These finite values are also reproduced by the independent verification described in Section~\ref{sec:computational-verification}.

\section{Independent computational verification}
\label{sec:computational-verification}

The theorems above are proved combinatorially.  Computation is used only as an independent consistency check and to certify the explicitly finite exceptional cases.  Because several formulas in the proof involve many interval endpoints, an implementation that follows a substantially different route is useful for detecting transcription errors, indexing shifts, or unnoticed boundary cases.

\subsection{A verifier that does not assume the answer}
\label{subsec:independent-verifier}

The independent program constructs $t_n$ directly from the parity of the binary digit sum of each position, rather than recursively reusing the routines from which the symbolic formulas were derived.  It then performs the following steps.

\begin{enumerate}[label=\arabic*.]
\item It enumerates every distinct nonempty factor of the ambient word $t_n$.
\item For each factor it finds all, possibly overlapping, occurrences by an ordinary exact string search.
\item It represents the coverage as the ordinary set of text positions belonging to at least one occurrence interval.
\item It computes the inclusion-minimal coverage antichain by direct pairwise set containment.
\item Independently, it searches for transversals of the coverage system up to cardinality four.
\end{enumerate}

The transversal search is exhaustive.  At a partial solution it selects a coverage not yet hit and branches over every position of that coverage.  Every completing transversal must choose at least one of those positions, so no solution can be lost.  Memoization merges identical search states, but no expected attractor family, expected count, or geometric restriction is used to prune a branch.  Expected values are compared with the output only after the enumeration is complete.

This distinction matters methodologically.  The verifier is not a computer-assisted proof of the uniform theorems; rather, it is a second implementation of the definitions that can falsify the symbolic derivations on all tested orders.

\subsection{Exhaustive finite data}
\label{subsec:verification-data}

The principal exhaustive computations are summarized below.  Every factor and every coverage is taken in the ambient word $t_n$.

\begin{table}[ht]
\centering
\caption{Independent exhaustive computations for the first four orders in the main range.}
\label{tab:independent-computations}
\begin{tabular}{rrrrr}
\toprule
$n$ & nonempty factors & minimal coverages & smallest attractors & certificate MUSs\\
\midrule
4 & 100   & 12 & 87 & 12\\
5 & 392   & 18 & 40 & 20\\
6 & 1,560 & 26 & 32 & 24\\
7 & 6,232 & 34 & 32 & 24\\
\bottomrule
\end{tabular}
\end{table}

For every row the search finds no transversal of size at most three.  Thus the computations independently recover the known lower bound $\gamma(t_n)\ge4$ in these cases.  For $t_4$, the $87$ exhaustive solutions are exactly the sets satisfying Proposition~\ref{prop:t4-extended}.  For $t_5$, equality is checked at the level of sets:
$
 \Acal(t_5)=\FL_5\cup\mathcal E_5\cup\FR_5,
$
not merely at the level of cardinalities.

The program also reconstructs the complete minimal-coverage antichains for $t_4$ and $t_5$ from scratch and verifies that they coincide with the tables in Appendix~\ref{app:finite-tables}.  This simultaneously checks the two small-order MUS certificates described above.

\subsection{Checks of the uniform ingredients}
\label{subsec:uniform-checks}

The independent implementation also exercises the structural ingredients outside the finite exceptional cases.

For every $4\le m\le8$ and every $r\ge0$ with $m+r+1\le10$, it verifies both the trimmed-image identity
$
 M_i^{(m+1)}=\operatorname{trim}(\mu(M_i^{(m)}))
$
and the full occurrence equality of the lifting lemma.  It checks the paired-boundary property directly against every factor of $t_4,\ldots,t_9$ and verifies every one of the $32$ proposed stable attractors against all factor coverages for $6\le n\le9$.

The canonical minimal-coverage theorem is tested independently for every $5\le n\le9$.  The program constructs the coverages of $aa$, $bb$, and all MUSs from generations $4$ through $n$ without being given the expected antichain.  It verifies that these coverages are distinct and pairwise incomparable, that their number is $8n-22$, and that the coverage of every factor of $t_n$ contains at least one of them.  The resulting minimal-coverage counts are
\[
 18,\quad26,\quad34,\quad42,\quad50
\]
for $n=5,6,7,8,9$, respectively.

The critical-descent mechanism is checked separately.  In every ambient order from four through nine the program recomputes all critical factors of lengths two through four; for every critical factor of length at least five through order nine it decodes the unique morphic phase, checks that the factor is a two-sided trim, and verifies that criticality descends to the decoded preimage.  It also checks directly the four short-factor dominance inclusions used in the proof of Lemma~\ref{lem:possible-minimal-representatives}.

Finally, for each $6\le n\le9$ the program constructs the counterexample
$
 \Gamma_n^\star=\{3h/2,3h,4h,6h\}
$
from Proposition~\ref{prop:depth-optimality}, recomputes all MUSs of $t_n,t_{n-1},t_{n-2}$, and confirms that every coverage from the first two generations is hit while $\Cov_{t_n}(M_3^{(n-2)})$ is missed.  It also checks the five-position counterexample showing that the stable three-generation certificate cannot be extended to arbitrary cardinalities.

\subsection{A second check of the interval geometry}
\label{subsec:parametric-check}

The sixteen-witness proof can be checked in a way that is independent of the enumeration of particular powers of two.  Write $h=4k$ and partition $[1,8h]$ at every interval start and at every interval end plus one.  Each event position is an affine integer function of $k$.  For $k\ge3$ their order is fixed, producing $49$ cells on which membership in all sixteen witness coverages is constant.

Exactly two products of four cells---one cell inside each of $A_2,B_7,B_2,A_7$---hit all sixteen witnesses.  Every selected cell contains two positions, so each product describes $2^4=16$ quadruples.  These are precisely the two stable families.  The remaining integer value $k=2$, corresponding to $h=8$, gives the same products when checked separately.

This parametric computation verifies the interval arithmetic of the sixteen-witness lemma for every multiple $h=4k\ge8$, not merely for the powers of two arising from Thue--Morse words.  The separate MUS-lifting check verifies that these intervals are indeed the factor coverages claimed in the proof.

\section{Discussion}
\label{sec:discussion-extended}

The results admit two complementary interpretations.  One concerns the geometry of optimum solutions; the other concerns the structure of the attractor condition itself.

\subsection{Why the number of smallest attractors stabilizes}
\label{subsec:why-32}

For $n\ge6$, every smallest attractor chooses one endpoint from the three mandatory boundary pairs
\[
 \{2h,2h+1\},\qquad
 \{4h,4h+1\},\qquad
 \{6h,6h+1\},
\]
and chooses its fourth position from exactly one of the reflected pairs
\[
 \{3h,3h+1\},
 \qquad
 \{5h,5h+1\}.
\]
Once these five two-position blocks are separated, increasing $n$ merely scales their distances.  It does not create new combinatorial choices.  This is why the number of smallest attractors becomes independent of the word length:
$
 2\cdot2^4=32.
$

The exceptional family $\mathcal E_5$ shows what fails one generation earlier.  At $n=5$ two central positions can still cooperate across a short gap in a way that is impossible once the relevant intervals have separated.  Thus the transition from $40$ to $32$ solutions is not an arbitrary small-order anomaly; it marks the point at which the recursive boundary geometry becomes fully separated.

\subsection{Canonical constraints versus optimum certificates}
\label{subsec:two-levels}

The minimal-coverage theorem reveals a different form of structure.  Without any cardinality restriction, the attractor condition has an exact canonical description using
$
 8n-22
$
pairwise incomparable constraints.  By Lemma~\ref{lem:canonical-irredundancy-extended}, none of them can be removed from an exact certificate drawn from factor coverages.

At optimum cardinality the situation changes drastically.  Four positions cannot hit the constraints independently: satisfying some of them already determines where the remaining positions must lie.  This extra combinatorial information collapses the linear-size canonical system to the $24$ MUS coverages of three consecutive generations, and only sixteen carefully chosen members are needed to force the final geometry.

It is useful to keep these two statements conceptually separate:
\[
 \boxed{\text{all attractors}}
 \Longleftrightarrow
 \boxed{8n-22\text{ canonical constraints}},
\]
whereas
\[
 \boxed{|\Gamma|=4}
 \Longrightarrow
 \boxed{24\text{ MUS constraints}}
 \Longrightarrow
 \boxed{16\text{ forcing witnesses}}.
\]
The first line describes the intrinsic constraint complexity of the word.  The second exploits the rigidity of a four-position transversal.

\subsection{What is specific to Thue--Morse?}
\label{subsec:specific-vs-portable}

Several ingredients of the proof use special features of the Thue--Morse morphism.  The explicit four-boundary geometry, the exact list of eight MUSs per generation, and the two stable attractor families are specific to this word.  The particularly simple trimmed recurrence
$
 M^{(m+1)}=\operatorname{trim}(\mu(M^{(m)}))
$
also depends on the binary uniform morphism and on the closed MUS description of Kishi et al.

Other parts of the argument are more portable.  The passage from arbitrary factor coverages to inclusion-minimal ones is completely general.  So is the canonical irredundancy lemma.  More suggestively, the Section~\ref{sec:canonical-coverages} strategy
\[
 \text{minimal coverage}
 \Rightarrow
 \text{frequency-critical representative}
 \Rightarrow
 \text{desubstitution}
\]
uses recognizability-like synchronization to reduce a global coverage question to a finite set of short factors.  This suggests a possible template for other uniformly morphic families whenever a sufficiently strong phase synchronization and a manageable description of critical factors are available.

Similarly, the certificate-depth viewpoint does not depend on the numerical value three.  For Thue--Morse, current-generation MUSs constrain candidate positions locally, while two older generations provide the missing global information.  This motivates studying, for other morphic families, the minimum number of consecutive source generations needed to certify smallest attractors and whether such a depth can be controlled by recognizability or occurrence-preservation properties of the underlying morphism.

These are directions suggested by the present proof rather than claims established here.  In particular, we do not assert that bounded-depth MUS certificates exist for every uniformly morphic family.

\subsection{The role of occurrence transport}
\label{subsec:role-occurrence-transport}

Occurrence transport is the technical bridge connecting the local MUS description with the global attractor geometry.  A MUS is unique only in its source word.  Once embedded in a later generation it may occur many times, and the attractor condition depends on the union of all of these occurrences.  The exact doubling law makes those ambient coverages computable from constant base data.

The point is therefore not merely that one can generate the next MUS from the previous one.  What matters is that the complete set of occurrence positions is transported exactly.  This is what allows both kinds of results in the paper: the linear-size canonical certificate for arbitrary attractors and the constant-size witness system for optimum four-position sets.

\subsection{Computation as verification rather than discovery in the proof}
\label{subsec:role-computation}

The final theorems are uniform and do not rely on checking a finite range of orders.  Computation nevertheless plays a useful supporting role.  It resolves the genuinely finite exceptional cases, gives an independent implementation of the definitions, and checks delicate interval and indexing calculations on a range large enough to expose the recursive pattern.

Keeping this role explicit is important.  The structural proof explains why the answer stabilizes, why the value is $32$, why three MUS generations are needed, and why the minimal-coverage count is $8n-22$.  Exhaustive computation alone would confirm these patterns for particular orders but would not supply those explanations.

\section{Conclusion}
\label{sec:conclusion-extended}

We have completely characterized the smallest string attractors of the finite Thue--Morse words.  The number of smallest attractors is
\[
 87,\quad40,\quad32,\quad32,\ldots
\]
from orders four onward.  For every $n\ge6$, the stable value $32$ arises from two explicit reflected families obtained by independently selecting endpoints of four recursive boundary pairs.

The classification is accompanied by a broader structural result.  For every $n\ge5$, the complete inclusion-minimal factor-coverage antichain consists of the coverages of $aa$, $bb$, and the eight MUSs of every source generation from four through $n$.  Its size is exactly
$
 8n-22,
$
and it is the unique irredundant exact certificate at the level of factor coverages for attractors of arbitrary cardinality.

Under the minimum-cardinality restriction this universal condition collapses.  The MUSs of three consecutive generations give an exact $24$-factor certificate, depth three is necessary within that hierarchy, and a distinguished set of sixteen witnesses already forces the geometry of every smallest attractor.

The resulting picture has two complementary scales: a canonical linear-size constraint system for arbitrary attractors and a constant-size certificate at optimum cardinality.  The recursive-boundary, occurrence-transport, and critical-desubstitution mechanisms behind these two scales are discussed in Section~\ref{sec:discussion-extended}.

\bibliographystyle{plain}
\bibliography{bibliography}

@inproceedings{Banbara2026,
  author       = {Banbara, Mutsunori and Bannai, Hideo and Guo, Peaker and K{\"o}ppl, Dominik and Mieno, Takuya and Okamoto, Yoshio},
  title        = {The Smallest String Attractors of Fibonacci and Period-Doubling Words},
  booktitle    = {37th Annual Symposium on Combinatorial Pattern Matching (CPM 2026)},
  editor       = {Bille, Philip and Prezza, Nicola},
  series       = {Leibniz International Proceedings in Informatics (LIPIcs)},
  volume       = {369},
  pages        = {33:1--33:21},
  publisher    = {Schloss Dagstuhl -- Leibniz-Zentrum f{\"u}r Informatik},
  year         = {2026},
  doi          = {10.4230/LIPIcs.CPM.2026.33},
  url          = {https://doi.org/10.4230/LIPIcs.CPM.2026.33}
}

@misc{Beal2025,
  author       = {B{\'e}al, Marie-Pierre and Crochemore, Maxime and Romana, Giuseppe},
  title        = {Checking and Producing Word Attractors},
  year         = {2025},
  eprint       = {2509.08503},
  archivePrefix = {arXiv},
  primaryClass = {cs.DS},
  url          = {https://arxiv.org/abs/2509.08503}
}

@article{Cassaigne2024,
  author       = {Cassaigne, Julien and Gheeraert, Florent and Restivo, Antonio and Romana, Giuseppe and Sciortino, Marinella and Stipulanti, Manon},
  title        = {New String Attractor-Based Complexities for Infinite Words},
  journal      = {Journal of Combinatorial Theory, Series A},
  volume       = {208},
  pages        = {105936},
  year         = {2024},
  doi          = {10.1016/j.jcta.2024.105936},
  url          = {https://doi.org/10.1016/j.jcta.2024.105936}
}

@inproceedings{Dolce2023,
  author       = {Dolce, Francesco},
  title        = {String Attractors for Factors of the Thue--Morse Word},
  booktitle    = {Combinatorics on Words (WORDS 2023)},
  series       = {Lecture Notes in Computer Science},
  volume       = {13899},
  pages        = {117--129},
  publisher    = {Springer},
  year         = {2023},
  doi          = {10.1007/978-3-031-33180-0_9},
  url          = {https://doi.org/10.1007/978-3-031-33180-0_9}
}

@article{IlieSmyth2011,
  author       = {Ilie, Lucian and Smyth, W. F.},
  title        = {Minimum Unique Substrings and Maximum Repeats},
  journal      = {Fundamenta Informaticae},
  volume       = {110},
  number       = {1--4},
  pages        = {183--195},
  year         = {2011},
  doi          = {10.3233/FI-2011-536},
  url          = {https://doi.org/10.3233/FI-2011-536}
}

@inproceedings{KempaPrezza2018,
  author       = {Kempa, Dominik and Prezza, Nicola},
  title        = {At the Roots of Dictionary Compression: String Attractors},
  booktitle    = {Proceedings of the 50th Annual ACM SIGACT Symposium on Theory of Computing (STOC 2018)},
  pages        = {827--840},
  publisher    = {ACM},
  year         = {2018},
  doi          = {10.1145/3188745.3188814},
  url          = {https://doi.org/10.1145/3188745.3188814}
}

@inproceedings{Kishi2026,
  author       = {Kishi, Kaisei and Guo, Peaker and Urbina, Cristian and Bannai, Hideo},
  title        = {On Occurrence-Preserving Morphisms},
  booktitle    = {37th Annual Symposium on Combinatorial Pattern Matching (CPM 2026)},
  editor       = {Bille, Philip and Prezza, Nicola},
  series       = {Leibniz International Proceedings in Informatics (LIPIcs)},
  volume       = {369},
  pages        = {24:1--24:16},
  publisher    = {Schloss Dagstuhl -- Leibniz-Zentrum f{\"u}r Informatik},
  year         = {2026},
  doi          = {10.4230/LIPIcs.CPM.2026.24},
  url          = {https://doi.org/10.4230/LIPIcs.CPM.2026.24}
}

@inproceedings{Kutsukake2020,
  author       = {Kutsukake, K. and Matsumoto, T. and Nakashima, Y. and Inenaga, S. and Bannai, H. and Takeda, M.},
  title        = {On Repetitiveness Measures of Thue--Morse Words},
  booktitle    = {String Processing and Information Retrieval (SPIRE 2020)},
  series       = {Lecture Notes in Computer Science},
  volume       = {12303},
  pages        = {213--220},
  publisher    = {Springer},
  year         = {2020},
  doi          = {10.1007/978-3-030-59212-7_15},
  url          = {https://doi.org/10.1007/978-3-030-59212-7_15}
}

@inproceedings{Mantaci2019,
  author       = {Mantaci, Sabrina and Restivo, Antonio and Romana, Giuseppe and Rosone, Giovanna and Sciortino, Marinella},
  title        = {String Attractors and Combinatorics on Words},
  booktitle    = {Proceedings of the 20th Italian Conference on Theoretical Computer Science (ICTCS 2019)},
  series       = {CEUR Workshop Proceedings},
  volume       = {2504},
  pages        = {57--71},
  year         = {2019},
  url          = {https://ceur-ws.org/Vol-2504/paper8.pdf}
}

@article{Mantaci2021,
  author       = {Mantaci, Sabrina and Restivo, Antonio and Romana, Giuseppe and Rosone, Giovanna and Sciortino, Marinella},
  title        = {A Combinatorial View on String Attractors},
  journal      = {Theoretical Computer Science},
  volume       = {850},
  pages        = {236--248},
  year         = {2021},
  doi          = {10.1016/j.tcs.2020.11.006},
  url          = {https://doi.org/10.1016/j.tcs.2020.11.006}
}

@misc{Mousavi2016,
  author       = {Mousavi, Hamoon},
  title        = {Automatic Theorem Proving in Walnut},
  year         = {2016},
  eprint       = {1603.06017},
  archivePrefix = {arXiv},
  primaryClass = {cs.FL},
  url          = {https://arxiv.org/abs/1603.06017}
}

@article{Schaeffer20,
  author       = {Jeffrey O. Shallit},
  title        = {String Attractors for Automatic Sequences},
  journal      = {CoRR},
  volume       = {abs/2012.06840},
  year         = {2020},
  url          = {https://doi.org/10.48550/arXiv.2012.06840},
  doi          = {10.48550/ARXIV.2012.06840},
  eprinttype   = {arXiv},
  eprint       = {2012.06840},
  bibsource    = {dblp computer science bibliography, https://dblp.org}
}

\clearpage 

\appendix
\section{Finite coverage tables for \texorpdfstring{$t_4$ and $t_5$}{t4 and t5}}
\label{app:finite-tables}

This appendix records the finite data underlying Propositions~\ref{prop:t4-extended} and~\ref{prop:t5-extended}.  Because Proposition~\ref{prop:minimal-coverage-reduction} reduces the attractor property to the inclusion-minimal coverage antichain, these tables are complete finite certificates rather than merely collections of necessary conditions.

\subsection{The twelve minimal coverages of \texorpdfstring{$t_4$}{t4}}

For $t_4$, there are $100$ distinct nonempty factors and $99$ distinct factor coverages.  Exactly twelve of the latter are inclusion-minimal.  They are the coverages of the eight MUSs of $t_4$ and the four relevant MUSs of $t_3$:

\begin{center}
\begin{tabular}{ll@{\qquad}ll}
\toprule
factor & coverage & factor & coverage\\
\midrule
$bbab$ & $[2,5]$ & $baba$ & $[3,6]$\\
$abaa$ & $[4,7]$ & $aabb$ & $[6,9]$\\
$bbaa$ & $[8,11]$ & $aaba$ & $[10,13]$\\
$abab$ & $[11,14]$ & $babb$ & $[12,15]$\\
$aa$ & $[6,7]\cup[10,11]$ & $bb$ & $[2,3]\cup[8,9]\cup[14,15]$\\
$aba$ & $[4,6]\cup[11,13]$ & $bab$ & $[3,5]\cup[12,14]$\\
\bottomrule
\end{tabular}
\end{center}

Every four-position hitting set of these twelve coverages contains one point in $[2,5]$, two in $[6,11]$, and one in $[12,15]$.  Expanding the twelve intersection predicates and removing redundant implications yields the ten clauses of Proposition~\ref{prop:t4-extended}.  Direct counting by the value of $\ell$ gives $10+21+28+28=87$ solutions.

\subsection{The eighteen minimal coverages of \texorpdfstring{$t_5$}{t5}}

For $t_5$, there are $392$ distinct nonempty factors and $391$ distinct factor coverages.  Exactly eighteen are inclusion-minimal.  The eight current-generation coverages are

\begin{center}
\begin{tabular}{ll@{\qquad}ll}
\toprule
factor & coverage & factor & coverage\\
\midrule
$abaabb$ & $[4,9]$ & $aabbaa$ & $[6,11]$\\
$bbaaba$ & $[8,13]$ & $babbab$ & $[12,17]$\\
$abaaba$ & $[16,21]$ & $babbaa$ & $[20,25]$\\
$bbaabb$ & $[22,27]$ & $aabbab$ & $[24,29]$\\
\bottomrule
\end{tabular}
\end{center}

The remaining ten minimal coverages are

\begin{center}
\begin{tabular}{ll}
\toprule
factor & coverage\\
\midrule
$aaba$ & $[10,13]\cup[18,21]$\\
$abab$ & $[11,14]\cup[19,22]$\\
$babb$ & $[12,15]\cup[20,23]$\\
$bbaa$ & $[8,11]\cup[22,25]$\\
$aabb$ & $[6,9]\cup[24,27]$\\
$bbab$ & $[2,5]\cup[14,17]\cup[26,29]$\\
$baba$ & $[3,6]\cup[15,18]\cup[27,30]$\\
$abaa$ & $[4,7]\cup[16,19]\cup[28,31]$\\
$aa$ & $[6,7]\cup[10,11]\cup[18,19]\cup[24,25]\cup[30,31]$\\
$bb$ & $[2,3]\cup[8,9]\cup[14,15]\cup[22,23]\cup[26,27]$\\
\bottomrule
\end{tabular}
\end{center}

Every four-position transversal of these eighteen masks contains one point in $\{8,9\}$ and one in $\{24,25\}$.  Its central pair belongs to one of the three classes
\[
 \{12,13\}\times\{16,17\},
 \qquad
 \{(14,18),(15,19)\},
 \qquad
 \{16,17\}\times\{20,21\}.
\]
The outer choices are independent.  The three classes therefore produce $16$, $8$, and $16$ attractors, respectively, giving the total $40$.

\section{Details of the independent verification}
\label{app:independent-verification}

The independent verifier is intentionally separate from the implementation used during the derivation of the symbolic formulas.  Its construction follows the definitions directly and does not import helper routines from the original program.

For each tested order it generates the Thue--Morse word, enumerates every distinct factor, computes all overlapping occurrences, builds the corresponding coverage set, removes duplicate coverage sets, and then filters the remaining family by inclusion.  For the transversal computation it branches only on positions belonging to a currently unhit constraint; this is an exhaustive hitting-set search, not a heuristic one.

Besides the values in Table~\ref{tab:independent-computations}, the verifier performs the following independent checks:
\begin{itemize}
\item the complete $t_4$ and $t_5$ antichains in Appendix~\ref{app:finite-tables};
\item equality of the exhaustive attractor sets with the descriptions in Propositions~\ref{prop:t4-extended} and~\ref{prop:t5-extended};
\item the trimmed MUS identity and full occurrence transport through ambient order ten;
\item the centered/paired-boundary property through $t_9$;
\item every proposed stable attractor against every distinct factor for $6\le n\le9$;
\item the minimal-coverage theorem for $5\le n\le9$;
\item the short critical-factor lists and critical descent through order nine;
\item the depth-two counterexample and the five-position limitation of the stable MUS certificate.
\end{itemize}

No random sampling is used in any of these tests.

\section{Indexing conventions in the cited Thue--Morse literature}
\label{app:indexing}

Two indexing conventions appear in the literature used in this paper.  Kutsukake et al.~\cite{Kutsukake2020} use the same convention as we do:
\[
 t_0=a,
 \qquad
 |t_n|=2^n.
\]
Kishi et al.~\cite{Kishi2026} write
$
 T_i=\mu^{i-1}(a),
$
so that
$
 T_i=t_{i-1}.
$
Consequently their formula for the MUSs of $T_i$ translates into our Theorem~\ref{thm:kishi-mus} after the substitution $m=i-1$.  The shift is also consistent with the lengths $|T_i|=2^{i-1}$ and $|t_n|=2^n$.

\end{document}